\documentclass[11pt,letterpaper]{article}
\usepackage[margin=1in]{geometry}

\usepackage{amsfonts, amsmath, amsthm, amssymb}
\usepackage{color}
\usepackage{microtype}
\usepackage[hidelinks]{hyperref}
\hypersetup{pdftitle={On Convergence to Equilibrium for the Boltzmann Hierarchy}, pdfauthor={Joseph K. Miller}}
\usepackage{mathrsfs}
\usepackage[nameinlink]{cleveref}

\newcommand{\Rbb}{\mathbb{R}}
\newcommand{\Sbb}{\mathbb{S}}
\newcommand{\Nbb}{\mathbb{N}}
\newcommand{\Xbb}{\mathbb{X}}
\newcommand{\Pbb}{\mathbb{P}}
\newcommand{\Tbb}{\mathbb{T}}
\newcommand{\Zbb}{\mathbb{Z}}
\newcommand{\Ecal}{\mathcal{E}}
\newcommand{\Hcal}{\mathcal{H}}
\newcommand{\Dcal}{\mathcal{D}}
\newcommand{\Sym}{\textrm{Sym}}

\newtheorem{lemma}{Lemma}
\newtheorem{theorem}{Theorem}
\newtheorem{proposition}{Proposition}
\newtheorem{corollary}{Corollary}
\theoremstyle{remark}

\theoremstyle{definition}
\newtheorem{definition}{Definition}
\newtheorem{assumption}{Assumption}

\title{On Convergence to Equilibrium for the Boltzmann Hierarchy}
\author{Joseph K. Miller\\
\small Massachusetts Institute of Technology, Cambridge, MA, USA\\
\small\href{mailto:jkm314@mit.edu}{\texttt{jkm314@mit.edu}}}
\date{October 16, 2026}

\begin{document}
\maketitle

\begin{abstract}
We study relaxation of correlated solutions of the hard-cutoff Boltzmann hierarchy toward factorized Maxwellian equilibrium. We prescribe the Maxwellian's mass, momentum, and energy in each full particle coordinate, conditional on the others. Centering removes these conserved modes, but collisions still couple each centered level to the next. To control this forcing uniformly in particle number, we combine adapted second-order Sobolev tensor norms with geometric activity weights and prove
\[
\frac12\frac{d}{dt}\Ecal_z(H(t))^2
+(\lambda-\sqrt{2}C_\Gamma z)\Dcal_z(H(t))^2\le0,
\]
where $H$ is the rescaled centered hierarchy, $\Ecal_z,\Dcal_z$ measure its energy and dissipation, and the coercivity and collision constants $\lambda,C_\Gamma$ are independent of particle number. For small activity $z$, dissipation absorbs the coupling and gives contraction and uniqueness in the hierarchy energy. Time-dependent weights yield level-proportional exponential decay. For global solutions, densities and observable generating functions converge to equilibrium, and connected pair correlations decay at twice the one-particle rate. The estimates apply to existing energy solutions. General existence remains an assumption, while statistical mixtures, including continuous laws, of sufficiently regular one-particle solutions uniformly close to the reference Maxwellian, with common conserved moments, provide examples.

\end{abstract}


\tableofcontents

\section{Introduction}\label{sec: intro}

We study the statistics of possibly correlated rarefied gases undergoing binary collisions. Our reference equilibrium describes independent particles with uniform positions on $\Tbb^3$ (with normalized spatial measure) and Maxwellian velocity density $\mu(v):=(2\pi)^{-3/2}e^{-|v|^2/2}$. Their joint density is the product $\mu(v_1)\cdots\mu(v_k)$, abbreviated $\mu^{\otimes k}$. We allow correlated initial states, whose joint densities need not factorize; relaxation must therefore be measured at the level of the joint statistics.

\emph{For correlated initial data close to the factorized Maxwellian equilibrium, under what conditions do the joint densities approach this equilibrium, and at what rates?}

To study this relaxation, we use the Boltzmann hierarchy. Its unknowns $f_k(t,x_1,v_1,\ldots,x_k,v_k)$, $k\ge1$, are joint probability densities for the positions and velocities of $k$ particles at time $t$, symmetric under relabeling identical particles. The one-particle density $f_1$ gives individual-particle statistics, while $f_2$ also records pair correlations. At level $k$, free transport is coupled to collisions between one of the $k$ particles and an additional particle:
\[
\partial_t f_k+\sum_{j=1}^k v_j\cdot\nabla_{x_j}f_k
=\sum_{j=1}^k Q_j^{k+1}[f_{k+1}],\qquad k\ge1.
\]
Here $Q_j^{k+1}$ is the collision operator for particles $j$ and $k+1$, with the partner evaluated at $x_j$ and its velocity integrated out.

The Boltzmann hierarchy arises formally from the BBGKY equations for particle marginals in the Boltzmann--Grad limit. Statistically independent particles have product joint densities, written $f_k=f^{\otimes k}$; for these factorized solutions, the hierarchy reduces to the Boltzmann equation
\[
\partial_t f+v\cdot\nabla_x f=Q[f\otimes f],
\]
where $Q$ is the one-particle collision integral acting on a pair density. With correlations, the equation for $f_1$ depends on $f_2$ and does not close at the one-particle level.

For asymptotically independent initial data, deriving the Boltzmann equation requires controlling microscopic collision correlations and recovering factorization in the limit. This is one aspect of Hilbert's sixth problem. Lanford's short-time hard-sphere derivation \cite{CIP1994} compares collision-history expansions of the BBGKY and Boltzmann hierarchies. Gallagher, Saint-Raymond, and Texier \cite{GST2014} develop this approach, treating a class of short-range potentials with truncated marginals and cluster estimates. Deng, Hani, and Ma \cite[Theorem~1 and Section~1.3.1]{DengHaniMa2024} extend the hard-sphere derivation to intervals on which the Boltzmann equation has a regular solution: they compare BBGKY expansions with Boltzmann-equation iterates, using cumulant estimates to control correlation errors over successive time intervals. Their limiting joint densities are the products $f^{\otimes k}$ generated by the Boltzmann equation, which form a factorized solution of the Boltzmann hierarchy. Here we study the limiting hierarchy directly, allowing correlated solutions.

Ukai \cite{Ukai1974} proved global classical solutions of the Boltzmann equation for small inhomogeneous Maxwellian perturbations. Mouhot--Neumann \cite{MN06} proved quantitative coercivity of the linearized collision and transport operator in adapted Sobolev norms after removing global invariant modes; see also \cite{CIP1994,Glassey96}.

We extend the near-equilibrium analysis of Esposito and Pulvirenti \cite{PE89} with a direct energy--dissipation estimate for the kernels of \Cref{A: B regularity} in adapted $H^2_{x,v}$ tensor norms. Their work gives global statistical solutions, uniqueness in a controlled class, and level-proportional exponential decay for periodic hard spheres. Our estimate controls existing energy solutions and their differences by the initial hierarchy energy in the same norm throughout evolution, without factorization or a statistical-mixture representation. Preserving initial regularity was discussed in \cite[Remark~2]{PE89}. Their weighted velocity-supremum and higher spatial Sobolev assumptions are not nested with ours.

\subsection{Main results}
We consider macroscopically admissible hierarchies: the mass, momentum, and energy integrated over each particle's full position--velocity coordinate match those of $\mu$, conditional on the remaining particles (\Cref{def: macro admissible}). Expanding around $\mu^{\otimes k}$ then removes the conserved modes. Write $H=(h_k)_{k\ge1}$ for the Maxwellian-rescaled centered coefficients (\cref{eqn: inversion formula,eqn: f k hat}).

Collisions force each level from the next, so the summed energy inequality requires a correlated-pair bound uniform in particle number. The energy norm $\|\cdot\|_{\Hcal_k}$ uses the adapted one-particle inner product in each slot; $\|\cdot\|_{\Dcal_k}$ combines one-slot dissipation (\cref{eqn: Hcal 2 k,eqn: Hs Lambda k}). Fix an activity parameter $z>0$ controlling growth across levels, and define
\[
\Ecal_z(H)^2:=\sum_{k\ge1}z^{-2k}\|h_k\|_{\Hcal_k}^2,
\qquad
\Dcal_z(H)^2:=\sum_{k\ge1}z^{-2k}\|h_k\|_{\Dcal_k}^2.
\]
Let $\lambda,c_0>0$ be the one-particle coercivity and dissipation-comparison constants, and $C_\Gamma>0$ the pair-collision constant. For the hard-cutoff kernels of \Cref{A: B regularity}, our central estimate is
\[
\frac12\frac{d}{dt}\Ecal_z(H(t))^2
+(\lambda-\sqrt{2}C_\Gamma z)\Dcal_z(H(t))^2\le0.
\]
This inequality holds in integrated form for every energy solution as defined in \Cref{def:Xz-solution}. Geometric weights leave a factor $z$ in the coupling, so dissipation absorbs the interaction when $0<z<z_\ast:=\lambda/(\sqrt{2}C_\Gamma)$. For factorized initial perturbations $h_k(0)=u^{\otimes k}$, with $u$ in the one-particle energy space $\Hcal$, the energy becomes
\[
\Ecal_z(H(0))^2=\sum_{k\ge1}(\|u\|_{\Hcal}/z)^{2k}.
\]
It is finite exactly when $\|u\|_{\Hcal}<z$. Thus small activity restricts the data, but the finite value of $\Ecal_z(H(0))$ itself need not be small.

With $\eta_z:=c_0(\lambda-\sqrt{2}C_\Gamma z)>0$, \Cref{thm: main thm,thm:theorem-2} give
\begin{align*}
\Ecal_z(H(t))&\le e^{-\eta_z t}\Ecal_z(H(0)),\\
\|h_k(t)\|_{\Hcal_k}&\le z^k e^{-k\eta_z t}\Ecal_z(H(0)),\qquad k\ge1.
\end{align*}
Both estimates hold throughout the existence interval and for differences of solutions, giving uniqueness. The stronger second bound follows by retaining the level-dependent dissipation through time-dependent weights.

For global macroscopically admissible solutions, these estimates describe equilibration in the original statistical variables. Reconstruction gives $f_k(t)\to\mu^{\otimes k}$ in the mixed Sobolev tensor norm, with an explicit bound at rate $\eta_z$ (\Cref{cor:fk-to-muk-kdep}). For a one-particle observable $\phi(x,v)$, the quantity $\langle\phi^{\otimes k},f_k(t)\rangle$ is the average of the product of $\phi$ evaluated at $k$ particles. The dynamical partition function collects these moments across all levels in the exponential generating series
\begin{align*}
\mathcal Z(\beta,\mu,\phi;F(t))
&:=\sum_{k=0}^\infty\frac{\beta^k}{k!}\langle\phi^{\otimes k},f_k(t)\rangle,\qquad f_0=1,\\
&\longrightarrow e^{\beta\langle\phi,\mu\rangle}\qquad(t\to\infty).
\end{align*}
For real $\beta$ and real measurable $\phi$ with $\|\phi\mu^{1/2}\|_{L^2}<\infty$, the series converges absolutely and its departure from equilibrium decays at rate $\eta_z$, with a bound proportional to $\Ecal_z(H(0))$ (\Cref{thm:Z-convergence}). The logarithm of $\mathcal Z$ is eventually well defined and converges to $\beta\langle\phi,\mu\rangle$ at the same rate.

For general initial hierarchies, existence remains an assumption. Correlated examples arise from statistical mixtures, including continuous laws, of sufficiently regular one-particle solutions uniformly close to the reference Maxwellian, each sharing its mass, momentum, and energy (\Cref{prop:mixture-existence}).

\Cref{sec:hierarchy-setup,sec: main theorems} introduce the centered hierarchy, solution class, and main results. \Cref{sec:operator-estimates,sec: Energy Inequality} establish the operator and energy estimates. The density and observable consequences follow in \Cref{sec:consequences}, with combinatorial identities in \Cref{sec: Combinatorial Lemmas}.

\section{The hierarchy and equilibrium-centered decomposition}\label{sec:hierarchy-setup}
Throughout the paper, $\Tbb^3=(\Rbb/2\pi\Zbb)^3$ carries normalized Haar measure $dx=(2\pi)^{-3}d^3x$, so $\int_{\Tbb^3}dx=1$. Velocity integration uses Lebesgue measure, and angular integration uses the usual surface measure on $\Sbb^2$. We write $dX_k=\prod_{j=1}^k dx_j$, $dV_k=\prod_{j=1}^k dv_j$, and $dZ_k=dX_k\,dV_k$. All spatial integrals and norms use this convention; in particular, the velocity Maxwellian $\mu$ also has unit mass on $\Tbb^3\times\Rbb^3$. All solution functions and Sobolev spaces are real-valued, and Hilbert tensor products are over $\Rbb$; the centered coefficients may have either sign. Fourier calculations use the usual complexification.

\subsection{Collision law and Maxwellian equilibrium}\label{sec: Boltzmann Equation}
We first recall the one-particle collision law, whose perturbation operators will reappear in the centered hierarchy. For background on the Boltzmann equation, see Villani's review \cite{Villani2002}. For a probability density $f(t)\in L^1(\Tbb_x^3\times\Rbb^3_v)$, the Boltzmann equation reads
\begin{equation}\label{eqn: Boltzmann equation}
    \partial_t f(t,x,v) + v\cdot \nabla_x f(t,x,v) = Q[f\otimes f](t,x,v).
\end{equation}
To include correlated densities, we regard the collision integral as a linear operator on pair functions. Initially the inputs are smooth with rapid velocity decay; the collision and trace estimates in \Cref{sec:operator-estimates} will extend the operator to the energy spaces. Set
\begin{equation}\label{eqn: Q def}
    Q[f_2](x,v) = \int_{\Rbb^3_{v_1}} \int_{\Sbb^{2}} B(v-v_1,\sigma) \bigl[ f_2(x,v^\sigma,x,v_1^\sigma) -f_2(x,v,x,v_1) \bigr]d\sigma \, dv_1.
\end{equation}
For $\sigma\in\Sbb^2$ and $v,v_1\in\Rbb^3$, the post-collisional velocities in this integral are
\begin{equation}\label{eqn: post collision config}
    \begin{pmatrix}
        v^\sigma \\
        v_1^\sigma
    \end{pmatrix} = \frac{1}{2}
    \begin{pmatrix}
        v+v_1 + |v- v_1|
        \sigma \\
        v+ v_1 - |v-v_1| \sigma
    \end{pmatrix},
\end{equation}
and the collision kernel is
\begin{equation}\label{eqn: B def}
    B(u, \sigma) = |u|^\gamma b\left( \frac{u}{|u|} \cdot \sigma\right), \qquad u\ne0.
\end{equation}
Here $b$ is the angular collision profile, evaluated at the scattering-angle cosine $\cos\theta=(u/|u|)\cdot\sigma$. We impose \Cref{A: B regularity} below; the value of $B$ at $u=0$ is immaterial to the collision integrals.
\begin{assumption}\label{A: B regularity}
    We assume $b\in\mathcal C^1([-1,1])$, $b\ge0$, $b\not\equiv0$, and $\gamma\in[0,1]$.
    This class consists of angular-cutoff hard potentials, including Maxwell molecules when $\gamma=0$ and hard spheres when $\gamma=1$ and $b$ is constant, up to normalization \cite{Grad1963}.
\end{assumption}
We study relaxation toward the spatially homogeneous equilibrium given by the standard Maxwellian
\begin{equation}\label{eqn: mu definition}
    \mu(v)  = \frac{1}{(2\pi)^{3/2}}e^{-|v|^2/2}.
\end{equation}

\subsection{One-particle perturbation operators}\label{subsec:perturbation-operators}
The rescaling $f=\mu+\widehat f\mu^{1/2}$ separates the linear relaxation toward $\mu$ from the interaction of perturbations. Substitution into \cref{eqn: Boltzmann equation} gives the exact equation
\begin{equation}\label{eqn: linearized Boltzmann equation}
    \partial_t \widehat{f} + v \cdot \nabla_x \widehat{f} = L \bigl[\widehat{f} \bigr] + \Gamma\bigl[\widehat{f} \otimes \widehat{f} \bigr].
\end{equation}
The terms with one Maxwellian factor define the linearized operator
\begin{equation}\label{eqn: L def}
    L\bigl[\widehat{f}\bigr](x,v)
    := \mu^{-1/2}(v)\left(
    Q\bigl[(\mu^{1/2}\widehat{f})\otimes\mu\bigr](x,v)
    +Q\bigl[\mu\otimes(\mu^{1/2}\widehat{f})\bigr](x,v)
    \right).
\end{equation}
The remaining term is expressed through the linear pair operator
\begin{equation}\label{eqn: Gamma def}
    \Gamma\bigl[\widehat{f}_2\bigr](x,v) := \mu^{-1/2}(v)\, Q\bigl[ \mu_2^{1/2} \widehat{f}_2 \bigr](x,v).
\end{equation}
Here $\mu_2(v_1,v_2)=\mu(v_1)\mu(v_2)$ and $Q$ is defined in \cref{eqn: Q def}. Although $\Gamma$ is linear in its pair-function argument, the map $h\mapsto\Gamma[h\otimes h]$ in \cref{eqn: linearized Boltzmann equation} is quadratic. The operator $\Gamma$ decomposes into the velocity collision operator $\Gamma_{\rm loc}$ and the spatial diagonal restriction $\text{Tr}_{1,2}$:
\begin{equation}
    \Gamma = \text{Tr}_{1,2} \;  \Gamma_{\rm loc}, \qquad \text{Tr}_{1,2} [G](x) = G(x,x).
\end{equation}
The standard decomposition of the linearized operator \cite{Hilbert1912KineticGasTheory,Carleman1933Boltzmann,Grad1963,Ukai1974,MN06} separates a compact operator $K$ on $L^2_v$ from a positive multiplication operator $\Lambda$:
\begin{equation}\label{eqn: L decomposition}
    L = K - \Lambda.
\end{equation}
The multiplication operator $\Lambda$ introduces the collision-frequency weight used in our dissipation norms:
\begin{equation}\label{eqn: nu def}
    \nu(v) :=  \int_{\Rbb_{v_1}^3} \int_{\Sbb^2} B(v-v_1, \sigma) \mu(v_1) \, d\sigma \, dv_1, \qquad \Lambda f (x,v) := \nu(v) f(x,v).
\end{equation}
Under \Cref{A: B regularity}, the collision frequency satisfies, with $\langle v\rangle=(1+|v|^2)^{1/2}$,
\begin{equation}\label{eqn: collision frequency bounds}
    0<\nu_0,\qquad c_\nu\langle v\rangle^\gamma\le\nu(v)\le C_\nu\langle v\rangle^\gamma,\qquad \nu(v)\ge\nu_0.
\end{equation}
The bounds \cref{eqn: collision frequency bounds} follow from direct computation under \Cref{A: B regularity}; see, e.g., \cite[Chapter~2]{CIP1994}.
Combining collisions with transport gives the one-particle generator
\begin{equation}\label{eqn:A operator def}
    A = L -v \cdot \nabla_x.
\end{equation}
Mouhot--Neumann \cite{MN06} make the decay generated by $A$ quantitative in an adapted Sobolev inner product after removing the spatially constant invariant modes; see also \cite{CIP1994,Glassey96} for the linearized theory. We will use this particular inner product when tensorizing the estimate.

\subsection{Density hierarchies and symmetrization}\label{subsec:density-hierarchies}
At the hierarchy level, the pair input to the collision operator is part of a density with additional particle coordinates. We work with symmetric probability densities at each level:
\begin{equation}\label{eqn: P k def}
    \Pbb_k := \Bigg\{ 0\leq f_k \in L^1\Bigl((\Tbb_x^3 \times \Rbb^3_v)^k\Bigr) \, \Bigg| \, \, \big\Vert f_k \big\Vert_{L^1} =1, \, \, \text{and} \, \, \pi f_k = f_k\, , \, \forall \pi \in S_k \Bigg\},
\end{equation}
where $S_k$ is the permutation group on $\{1,\dots,k\}$, acting by
\begin{equation}\label{eqn: permutation action}
    \pi f_k( x_1, v_1, \ldots, x_k,v_k) = f_k( x_{\pi^{-1}(1)}, v_{\pi^{-1}(1)}, \ldots, x_{\pi^{-1}(k)},v_{\pi^{-1}(k)} ).
\end{equation}
We abbreviate the particle configurations by
\begin{equation*}
    X_k = (x_1, \dots, x_k), \qquad V_k=(v_1, \dots, v_k).
\end{equation*}
Products of lower-order symmetric functions need not be symmetric in all particle labels. The following average restores that symmetry and will let us combine centered coefficients with Maxwellian factors.
\begin{definition}\label{def: Sym operator}
    Let $k \geq 1$ and let $f_k = f_k(x_1, v_1, \dots, x_k, v_k)$ be a function on $(\Tbb^3 \times \Rbb^3)^k$. The symmetrization operator $\Sym$ is defined by
    \[
        (\Sym f_k)(X_k, V_k)
        =
        \frac{1}{k!}
        \sum_{\sigma \in S_k}
        \sigma f_k(X_k, V_k),
    \]
    where $S_k$ denotes the symmetric group on $k$ elements.
\end{definition}

We will study the Boltzmann hierarchy
\begin{equation}\label{eqn: boltzmann hierarchy}
    \partial_t f_k + \sum_{j=1}^k v_j \cdot \nabla_{x_j} f_k = \sum_{j=1}^k Q_j^{k+1} [f_{k+1}], \qquad k =1, 2, \dots.
\end{equation}
For each $j\in\{1,\ldots,k\}$, the collision operator acts on a $(k+1)$-particle function $H$ by
\begin{equation}\label{eqn: hierarchy collision operator}
\begin{aligned}
(Q_j^{k+1}H)(X_k,V_k)
={}&\int_{\Rbb^3}\int_{\Sbb^2}B(v_j-v_*,\sigma)
\Bigl[
H(X_k,x_j;V_k^{j\to v_j^\sigma},v_*^\sigma)\\
&\hspace{40mm}-H(X_k,x_j;V_k,v_*)
\Bigr]\,d\sigma\,dv_*.
\end{aligned}
\end{equation}
Here $V_k^{j\to w}$ is $V_k$ with its $j$th entry replaced by $w$, and $(v_j^\sigma,v_*^\sigma)$ is the post-collisional pair from \cref{eqn: post collision config} applied to $(v_j,v_*)$. The partner position is evaluated at $x_{k+1}=x_j$; only its velocity is integrated. The corresponding equilibrium hierarchy is the factorized sequence
\begin{equation}\label{eqn: mu k definition}
    \mu_k = \mu^{\otimes k}.
\end{equation}

\subsection{The centered decomposition and its inverse}\label{subsec:centered-decomposition}
To separate deviations from equilibrium at each order, we expand over the particles carrying a centered coefficient and place Maxwellian factors in the remaining slots. With $f_0=g_0=\mu_0=1$, this gives the centered decomposition\footnote{The author independently derived this symmetrized formula before learning of the equivalent subset-convolution formulation of Esposito and Pulvirenti \cite[Section~3, (3.5)--(3.10)]{PE89}.} for each $k\ge1$:
\begin{equation}\label{eqn: relative cumulant}
    f_k = \sum_{p = 0}^k \binom{k}{p} \Sym\left( g_p \otimes \mu_{k-p} \right).
\end{equation}
The first two coefficients are
\[
g_1=f_1-\mu,\qquad
g_2=f_2-f_1\otimes\mu-\mu\otimes f_1+\mu\otimes\mu.
\]
Thus $g_2$ remains after subtracting the equilibrium and one-particle corrections. The general inverse below uses the same inclusion--exclusion over subsets of particle labels: $\Sym$ averages subsets of a given size, and the binomial factor counts them. This Boolean-lattice inversion differs from the partition-lattice moment--cumulant relations and tensor treatment of generalized $k$-statistics described by Speed \cite{Speed1983,Speed1986}. Connected correlations arise from the logarithm of the generating function in \Cref{rem:gluing}.
\begin{proposition}\label{prop: Mobius inversion}
    The decomposition \cref{eqn: relative cumulant} uniquely determines $G=(g_k)_{k\ge1}$ from $F=(f_k)_{k\ge1}$ through the M\"obius inversion
    \begin{equation}\label{eqn: inversion formula}
    g_k = \sum_{p=0}^k (-1)^{k-p} \binom{k}{p} \Sym\left( f_p \otimes \mu_{k-p}\right).
\end{equation}
\end{proposition}

Writing the symmetrization as a sum over subsets records which particles carry centered coefficients. This form makes the collision expansion and the moment cancellations explicit.

\begin{lemma}\label{lem:Sym-subset-representation}
Let $k \ge 1$ and $0 \le p \le k$. 
For any symmetric function $h_p = h_p\bigl((x_i,v_i)_{i=1}^p\bigr)$ on $(\Tbb^3 \times \Rbb^3)^p$, define
\[
\mu_{k-p}(V_{k-p}) := \prod_{i=1}^{k-p} \mu(v_i).
\]
Then the symmetrization operator satisfies
\begin{equation}\label{eq:Sym-subset}
\Sym\bigl(h_p \otimes \mu_{k-p}\bigr)(X_k,V_k)
=
\frac{1}{\binom{k}{p}}
\sum_{\substack{S \subset \{1,\dots,k\}\\ |S|=p}}
h_p\bigl((x_i,v_i)_{i\in S}\bigr)
\prod_{i\notin S} \mu(v_i).
\end{equation}
\end{lemma}

The proofs of \Cref{prop: Mobius inversion} and \Cref{lem:Sym-subset-representation} are given in \Cref{sec: Combinatorial Lemmas}. These finite identities allow us to derive the transformed evolution directly at each level.

\subsection{Derivation of the centered hierarchy}\label{subsec:centered-evolution}
Applying the centered transformation to a solution $F=(f_k)_{k\ge1}$ of \cref{eqn: boltzmann hierarchy} separates its collision term into interactions with the Maxwellian background and a coupling to the next centered level:
\begin{equation}\label{eqn: g hierarchy}
\begin{aligned}
    \partial_t g_k+\sum_{j=1}^k v_j\cdot\nabla_{x_j}g_k
    ={}&\sum_{j=1}^k Q_j^{k+1}\!\left[
    g_k(Z_{[k]})\mu(v_{k+1})
    +\mu(v_j)g_k(Z_{[k]\setminus\{j\}},z_{k+1})
    \right]\\
    &+\sum_{j=1}^k Q_j^{k+1}[g_{k+1}],\qquad k\ge1.
\end{aligned}
\end{equation}
Here $[k]=\{1,\ldots,k\}$, $z_i=(x_i,v_i)$, and $Z_I=(z_i)_{i\in I}$, with indices in increasing order. Symmetry of $g_k$ makes the order of its arguments immaterial. The two background terms in \cref{eqn: g hierarchy} correspond to placing the Maxwellian in either colliding slot. The subset expansion makes this separation exact, as the following calculation shows.

\begin{lemma}\label{lem: g hierarchy derivation}
    Suppose a symmetric sequence $F$ solves the Boltzmann hierarchy \cref{eqn: boltzmann hierarchy}, with sufficient smoothness and velocity decay for the collision terms below to be defined. Define $G=(g_k)_{k\ge1}$ by \cref{eqn: relative cumulant}. Then $G$ satisfies \cref{eqn: g hierarchy}. The identity also holds distributionally whenever the corresponding collision terms are defined, since the calculation at each level involves only finite linear combinations.
\end{lemma}
\begin{proof}[Proof of \Cref{lem: g hierarchy derivation}]
Set $[k]=\{1,\ldots,k\}$ and use $*$ for the collision partner. For a finite set of labels $S$, write
\[
g_S:=g_{|S|}(Z_S),\qquad
\mu_S:=\prod_{i\in S}\mu(v_i),\qquad
D_S:=\partial_t+\sum_{j\in S}v_j\cdot\nabla_{x_j},
\]
with $g_\varnothing=\mu_\varnothing=1$. The subset form of the forward decomposition is
\[
f_k=\sum_{S\subset[k]}g_S\,\mu_{[k]\setminus S}.
\]
Because the Maxwellian is independent of $t$ and $x$, the total transport operator satisfies
\[
D_{[k]}f_k=\sum_{S\subset[k]}(D_Sg_S)\,\mu_{[k]\setminus S}.
\]

Let $Q_{j,*}$ denote \cref{eqn: hierarchy collision operator} with the other particle labels held fixed. Define
\[
\mathcal C_S:=\sum_{j\in S}Q_{j,*}\!\left[
g_S\mu(v_*)+\mu(v_j)g_{(S\setminus\{j\})\cup\{*\}}
+g_{S\cup\{*\}}
\right],\qquad \mathcal C_\varnothing:=0.
\]
We claim that the collision side of the density hierarchy has the same subset decomposition:
\[
\sum_{j=1}^k Q_{j,*}[f_{k+1}]
=\sum_{S\subset[k]}\mathcal C_S\,\mu_{[k]\setminus S}.
\]
For a fixed collision label $j$, group the subset expansion according to the presence of $j$ and $*$:
\[
\begin{aligned}
f_{k+1}=\sum_{R\subset[k]\setminus\{j\}}
\mu_{([k]\setminus\{j\})\setminus R}\Bigl[
&g_R\mu(v_j)\mu(v_*)+g_{R\cup\{j\}}\mu(v_*)\\
&+\mu(v_j)g_{R\cup\{*\}}+g_{R\cup\{j,*\}}\Bigr].
\end{aligned}
\]
The prefactor and $g_R$ are independent of the colliding variables, so the first term vanishes under $Q_{j,*}$ by $Q[\mu\otimes\mu]=0$. Setting $S=R\cup\{j\}$ in the other three terms gives the $j$th contribution to $\mathcal C_S\mu_{[k]\setminus S}$. Summing over $j$ counts every pair $(S,j)$ with $j\in S$ once, proving the claim.

The density hierarchy therefore gives
\[
\sum_{S\subset[k]}(D_Sg_S-\mathcal C_S)\mu_{[k]\setminus S}=0
\qquad\text{for every }k\ge1.
\]
Invertibility of the forward decomposition (\Cref{prop: Mobius inversion}) yields $D_{[k]}g_{[k]}=\mathcal C_{[k]}$, which is \cref{eqn: g hierarchy}.

In particular, at $k=1$ this reads
\[
(\partial_t+v_1\cdot\nabla_{x_1})g_1
=Q_1^2[g_1(z_1)\mu(v_2)]
 +Q_1^2[\mu(v_1)g_1(z_2)]
 +Q_1^2[g_2].
\]
Both background orientations survive; neither is removed by collision symmetry.
\end{proof}

The two background terms have the same structure as the linear terms in the one-particle perturbation equation. To recover the operators $L$ and $\Gamma$ exactly, rescale the centered coefficients by
\begin{equation}\label{eqn: f k hat}
    \widehat{f}_k  = \mu_k^{-1/2} g_k.
\end{equation}
The resulting hierarchy for $\widehat F=(\widehat f_k)_{k\ge1}$ is
\begin{equation}\label{eqn: f hat hieararchy}
    \partial_t \widehat{f}_k + \sum_{j=1}^k v_j \cdot \nabla_{x_j} \widehat{f}_k = \sum_{j=1}^k L_j \bigl[ \widehat{f}_k\bigr] + \sum_{j=1}^k \Gamma^{k+1}_j\bigl[ \widehat{f}_{k+1} \bigr].
\end{equation}
Here $L_j$ is the full operator $L$ acting in $(x_j,v_j)$, while $\Gamma_j^{k+1}$ applies $\Gamma$ to the pair $(j,k+1)$. Thus centering and rescaling preserve the linear hierarchy structure while exposing the one-particle dissipation.

To express this structure compactly, lift \cref{eqn: L decomposition} to each coordinate:
\begin{equation}\label{eqn: L j decomposition}
    L_j := K_j - \Lambda_j, \qquad \forall j =1, 2, \dots,
\end{equation}
where $\nu$ is defined in \cref{eqn: nu def}, $\Lambda_j[h](v_j):=\nu(v_j)h(v_j)$ is multiplication by the collision frequency on the $j$th component, and $K_j$ applies the velocity operator $K$ in the $j$th slot. The corresponding transport-collision operator is
\begin{equation}\label{eqn: Lambda j decomposition}
    A_j := L_j - v_j \cdot \nabla_{x_j}, \qquad \forall j = 1,2, \dots.
\end{equation}

Summing the one-particle generators and pair couplings, set
\begin{equation}\label{eqn: k operator def}
    A^k  = \sum_{j=1}^k A_j, \qquad K^k = \sum_{j=1}^k K_j, \qquad  \Gamma^{k+1} = \sum_{j=1}^k \Gamma_{j}^{k+1} .
\end{equation}
The superscripts label hierarchy levels, not powers. With the summed operators, \cref{eqn: f hat hieararchy} becomes
\begin{equation}\label{eqn: reduced f hat hierachy}
    \partial_t \widehat{f}_k = A^k \left[ \widehat{f}_k \right] + \Gamma^{k+1} \left[ \widehat{f}_{k+1}\right].
\end{equation}

\section{Solution class and precise main results}\label{sec: main theorems}
Write $h_k:=\widehat f_k=\mu_k^{-1/2}g_k$ and $H=(h_k)_{k\ge1}$. To use one-particle decay in every slot, we must remove the conserved modes and choose norms compatible with both transport and collisions. We first translate the statistical moment conditions into projection constraints, then incorporate them into the energy spaces and solution class.

\subsection{Admissibility and invariant projections}\label{subsec:admissibility}
To describe a single statistical state, the densities at different levels must be consistent: integrating out one particle must recover the preceding density. This is the admissibility condition used below.
\begin{definition}\label{def: Pbb}
    The space of \emph{admissible density hierarchies} is
\begin{equation}\label{eqn: P def}
    \Pbb := \left\{ \left(f_k\right)_{k=1}^\infty \in \prod_{k=1}^\infty \Pbb_k \Bigg\vert \, \, \int_{\Tbb^3\times \Rbb^3}f_{k+1}(X_k, x_{k+1} , V_k, v_{k+1}) dx_{k+1} dv_{k+1} = f_k(X_k, V_k), \qquad \forall k \geq 1\right\}.
\end{equation}
\end{definition}

At each fixed time, the Hewitt--Savage theorem \cite{HewittSavage55} represents an admissible hierarchy as a mixture of product probability measures on $\Tbb^3\times\Rbb^3$. This representation alone gives neither regular directing densities nor a mixing law transported by a regular Boltzmann flow. Our energy estimates do not use the representation and also apply to signed centered hierarchies and differences of solutions.

Consistency and first-marginal moments alone do not select the Maxwellian equilibrium. For $M_\vartheta(v):=(2\pi\vartheta)^{-3/2}e^{-|v|^2/(2\vartheta)}$ and $0<\delta<1$, the stationary hierarchy $f_k=\tfrac12(M_{1-\delta}^{\otimes k}+M_{1+\delta}^{\otimes k})$ has first-marginal mass $1$, momentum $0$, and second velocity moment $3$, but differs from $\mu^{\otimes k}$ and already violates the conditional energy identity at $k=2$. We therefore require the invariant moments in each particle coordinate, conditional on the remaining coordinates, to agree with those of $\mu$.
\begin{definition}\label{def: macro admissible}
Let $\phi(v)$ denote any collision invariant of the Boltzmann operator,
that is,
\[
\phi(v) \in \mathrm{span}\{1,v^1,v^2,v^3,|v|^2\}.
\]

A sequence $(f_k)_{k \ge 1} \in \Pbb$ is said to be
\emph{macroscopically admissible} if for every $k \ge 1$,
every particle index $j \in \{1,\dots,k\}$, and every collision
invariant $\phi$, one has
\begin{equation}
\label{eq:Jphi-admissibility}
\int_{\Tbb^3 \times \Rbb^3}
\phi(v_j)\, f_k(X_k,V_k)\, dx_j dv_j
=
\left(\int_{\Rbb^3} \phi(v)\mu(v)\, dv\right)
f_{k-1}(\widehat{X}_k^j, \widehat{V}_k^j),
\end{equation}
where $(\widehat{X}_k^j,\widehat{V}_k^j)$ denotes the $(k-1)$-particle configuration obtained by removing the $j$th variable, and $f_0:=1$ when $k=1$. The identity is understood almost everywhere in the remaining variables. We denote this set of sequences by $\Pbb_{macro}$. The macroscopic admissibility condition prescribes moments integrated over the full phase-space coordinate $(x_j,v_j)$, conditional on the remaining coordinates; it is not a pointwise-in-$x_j$ moment condition.
\end{definition}

We now translate these moment conditions into invariant projections for the perturbation variables.
\begin{definition}\label{def: pde null space}
For $L$ defined in \cref{eqn: L def}, let $\mathcal N$ denote its null space in velocity:
\begin{equation}\label{eqn: pde null space}
\mathcal N:=\ker_{L^2(\Rbb^3_v)}L
=\operatorname{span}\{\varphi_i\}_{i=1}^5,
\qquad
\begin{aligned}
\varphi_1(v)&=\mu^{1/2}(v),\\
\varphi_{a+1}(v)&=v^a\mu^{1/2}(v),\quad a=1,2,3,\\
\varphi_5(v)&=\frac{|v|^2-3}{\sqrt6}\mu^{1/2}(v).
\end{aligned}
\end{equation}
Here $v^a$ is the $a$th Cartesian component of $v$. These functions are orthonormal in $L^2_v$; in particular, $\int_{\Rbb^3}(|v|^2-3)^2\mu(v)\,dv=6$.

The velocity projection acts at each position:
\begin{equation}\label{eqn: velocity projection}
(\Pi_v h)(x,v)
:=\sum_{i=1}^5\left(\int_{\Rbb^3}\varphi_i(w)h(x,w)\,dw\right)\varphi_i(v).
\end{equation}
As an operator on $L^2_{x,v}$, $L$ has null space
\[
\left\{\sum_{i=1}^5a_i(x)\varphi_i(v):a_i\in L^2(\Tbb^3)\right\},
\]
the range of $\Pi_v$. By contrast, $A=L-v\cdot\nabla_x$ has the spatially constant null modes
\[
\ker A=\left\{\sum_{i=1}^5c_i\varphi_i(v):c_i\in\Rbb\right\}.
\]
The global $L^2_{x,v}$-orthogonal projection onto $\ker A$ is
\begin{equation}\label{eqn: global projection}
(\Pi_g h)(x,v)
:=\sum_{i=1}^5\left(\int_{\Tbb^3\times\Rbb^3}
\varphi_i(w)h(y,w)\,dy\,dw\right)\varphi_i(v)
=\int_{\Tbb^3}(\Pi_v h)(y,v)\,dy.
\end{equation}
Its output is independent of $x$. The normalized spatial measure makes the same basis orthonormal in $L^2_{x,v}$.
\end{definition}

The distinction is essential: $\Pi_g h=0$ requires only the five global invariant moments to vanish, whereas $\Pi_v h=0$ requires their vanishing at each position. For example, $h(x,v)=\sin(x^1)\mu^{1/2}(v)$ satisfies $\Pi_g h=0$ but $\Pi_v h=h$.

At the hierarchy level, apply these projections in each particle coordinate while holding the others fixed.
\begin{definition}\label{def: hierarchy null space}
Fix $k\ge1$ and $j\in\{1,\ldots,k\}$. The projection $\Pi_v^{(j)}$ applies \cref{eqn: velocity projection} in the $j$th velocity variable, holding every position and the other velocities fixed:
\begin{equation}\label{eqn: hierarchy velocity projection j}
(\Pi_v^{(j)}h_k)(X_k,V_k)
:=\sum_{i=1}^5\left(\int_{\Rbb^3}\varphi_i(w)
h_k(X_k,V_k^{j\to w})\,dw\right)\varphi_i(v_j).
\end{equation}
The range of $\Pi_v^{(j)}$ is $\ker L_j$; the invariant-mode coefficients may depend on all positions and the other velocities.

The projection $\Pi_g^{(j)}$ applies \cref{eqn: global projection} in the full $j$th phase-space coordinate:
\begin{equation}\label{eqn: hierarchy null projection j}
(\Pi_g^{(j)}h_k)(X_k,V_k)
:=\sum_{i=1}^5\left(\int_{\Tbb^3\times\Rbb^3}\varphi_i(w)
h_k(X_k^{j\to y},V_k^{j\to w})\,dy\,dw\right)\varphi_i(v_j).
\end{equation}
Here $X_k^{j\to y}$ replaces $x_j$ by $y$, and $V_k^{j\to w}$ replaces $v_j$ by $w$. The range of $\Pi_g^{(j)}$ is $\ker A_j$, where $A_j=L_j-v_j\cdot\nabla_{x_j}$. The coefficients may depend on the remaining particle coordinates, but are independent of $x_j$ and $v_j$. Both $\Pi_v^{(j)}$ and $\Pi_g^{(j)}$ are orthogonal projections in the full $L^2$ space.
\end{definition}

\subsection{One-particle spaces and the adapted norm}
The collision frequency supplies the velocity weight in the dissipation norm. For a positive velocity weight $w$, write $\|u\|_{L^2(w)}^2:=\int w(v)|u(x,v)|^2\,dx\,dv$; the subscript $v$ indicates velocity integration only.
For multiindices $\alpha,\beta\in\Nbb_0^3$, write $|\alpha|=\alpha_1+\alpha_2+\alpha_3$ and let $e_i$ be the $i$th coordinate vector. For an integer $s\ge0$, set
\begin{equation}\label{eqn: H s def}
\|u\|_{H^s}^2:=\sum_{|\alpha|+|\beta|\le s}\|\partial_x^\alpha\partial_v^\beta u\|_{L^2}^2,
\end{equation}
and define
\begin{equation}\label{eqn: Hs Lambda 1}
\|u\|_{H^s(\nu)}^2:=\sum_{|\alpha|+|\beta|\le s}\int\nu(v)|\partial_x^\alpha\partial_v^\beta u|^2\,dx\,dv.
\end{equation}
The adapted norm couples spatial and velocity derivatives so that collision dissipation also controls relaxation under transport. Fix positive constants $a_0,a_x,a_v,a_{xv},a_2,b_2,c_2,K_2$ by the Mouhot--Neumann construction \cite[Theorem~3.1]{MN06}, chosen so that $a_{xv}^2<a_xa_v$, $c_2^2<a_2b_2$, $K_2a_0\ge1$, and the coercivity conclusion of \Cref{prop: Lambda K L hypothesis} holds. With these coefficients fixed, set
\begin{equation}\label{eqn: Hcal 1}
\Phi_1(u):=a_0\|u\|_2^2+a_x\|\nabla_xu\|_2^2+a_v\|\nabla_vu\|_2^2
+2a_{xv}\langle\nabla_xu,\nabla_vu\rangle_2.
\end{equation}
For $|\alpha|=2$ and $\alpha_i>0$, put
\[
\Phi_{2,\alpha i}(u):=a_2\|\partial_x^\alpha u\|_2^2
+b_2\|\partial_x^{\alpha-e_i}\partial_{v_i}u\|_2^2
+2c_2\langle\partial_x^{\alpha-e_i}\partial_{v_i}u,\partial_x^\alpha u\rangle_2,
\]
so that the second-order norm is
\begin{equation}\label{eqn: Hcal 2}
\|u\|_{\mathscr H}^2
:=K_2 \Phi_1(u)+\sum_{|\beta|=2}\|\partial_v^\beta u\|_2^2
+\sum_{|\alpha|=2}\sum_{i:\alpha_i>0}\Phi_{2,\alpha i}(u).
\end{equation}
The normalization gives $\|u\|_{L^2}\le\|u\|_{\mathscr H}$. Real polarization of this quadratic form defines the inner product. For a multiindex pair $a=(\alpha,\beta)$, write $\partial^a:=\partial_x^\alpha\partial_v^\beta$. Polarizing and expanding \cref{eqn: Hcal 2} gives fixed real coefficients $m_{ab}=m_{ba}$, independent of particle number, such that
\begin{equation}\label{eq:derivative-form}
\langle u,w\rangle_{\mathscr H}
=\sum_{a,b\in I_2}m_{ab}\langle\partial^a u,\partial^b w\rangle_{L^2},
\qquad I_2:=\{(\alpha,\beta):|\alpha|+|\beta|\le2\}.
\end{equation}
Restricting to the complement of the global invariant modes gives the one-particle energy and dissipation spaces
\begin{equation}\label{eq:one-particle-spaces}
\Hcal:=(\ker\Pi_g\cap H^2,\|\cdot\|_{\mathscr H}),\qquad
\Dcal:=\ker\Pi_g\cap H^2(\nu),\qquad \|u\|_{\Dcal}:=\|u\|_{H^2(\nu)}.
\end{equation}
The ambient spaces without the projection constraint are denoted by $\mathscr H$ and $\mathscr D=H^2(\nu)$. For real Hilbert spaces $U_1,\ldots,U_k$, the Hilbert tensor product $U_1\otimes\cdots\otimes U_k$ is the completion of the algebraic tensor product for the inner product
\[
\bigl\langle u_1\otimes\cdots\otimes u_k,\,
w_1\otimes\cdots\otimes w_k\bigr\rangle
:=\prod_{j=1}^k\langle u_j,w_j\rangle_{U_j},
\]
extended bilinearly to finite sums of elementary tensors. For function spaces, $(u\otimes v)(z_1,z_2)=u(z_1)v(z_2)$. Thus $\otimes$ denotes the completed Hilbert tensor product when applied to spaces and the elementary product when applied to functions. The completed spaces also contain correlated functions.

\subsection{Tensor spaces and geometric activity norms}\label{subsec: hierarchy spaces}
Because $A^k$ is a sum of one-particle generators, the coercivity estimate leads us to sum the squared one-slot dissipation norms. We retain the adapted inner product in every other slot when forming the symmetric tensor spaces.
\begin{definition}\label{def: Mouhot Neumann Villani hierarchy norm}
For $k\ge1$ define
\begin{equation}\label{eqn: Hcal 2 k}
\Hcal_k:=\Sym(\Hcal^{\otimes k}),\qquad
D_{k,j}(u):=\|u\|_{\Hcal^{\otimes(j-1)}\otimes\Dcal\otimes\Hcal^{\otimes(k-j)}},
\end{equation}
and equip the symmetric intersection of these $k$ spaces with norm
\begin{equation}\label{eqn: Hs Lambda k}
\|u\|_{\Dcal_k}^2:=\sum_{j=1}^kD_{k,j}(u)^2.
\end{equation}
The tensor inner product is denoted by $\langle\cdot,\cdot\rangle_{\Hcal_k}$.
\end{definition}

We write $\Dcal_k^*$ for the space of continuous linear functionals
on $\Dcal_k$, with norm
\[
\|\ell\|_{\Dcal_k^*}
:=\sup_{0\ne v\in\Dcal_k}
\frac{|\ell(v)|}{\|v\|_{\Dcal_k}}.
\]
The energy space is the pivot: a function $g\in\Hcal_k$ acts by
$v\mapsto\langle g,v\rangle_{\Hcal_k}$, giving
$\Dcal_k\subset\Hcal_k\subset\Dcal_k^*$.
For collision outputs and smooth-core residuals, the dual action is
the continuous extension of this adapted energy pairing. This is the
duality used in the energy estimates; the hierarchy equation is
separately required in the usual distributional sense.
The one-particle notation $\Dcal^*$ has the same meaning.

These tensor spaces control derivatives separately in each particle coordinate. The next comparison identifies this regularity and records its dependence on $k$.

\begin{lemma}\label{lem:H2-vs-Hcal2}
There is $C\ge1$, independent of $k$, such that
\begin{equation}\label{eq:H2-Hcal2-equivalence}
C^{-k}\|u\|_{\mathcal S_k}\le\|u\|_{\Hcal_k}\le C^k\|u\|_{\mathcal S_k},\qquad
\mathcal S_k:=(H^2)^{\otimes k},\quad u\in\Hcal_k.
\end{equation}
The Sobolev tensor space $\mathcal S_k$ has squared norm $\|u\|_{\mathcal S_k}^2$ equal to the sum of the squared $L^2$ norms of all derivatives
$\prod_{j=1}^k\partial_{x_j}^{\alpha_j}\partial_{v_j}^{\beta_j}$ with $|\alpha_j|+|\beta_j|\le2$ for every $j$. In particular, it includes products of derivatives across particle coordinates and is not the Sobolev norm of total order two in $6k$ variables.
\end{lemma}
\begin{proof}
Let $S=H^2(\Tbb^3\times\Rbb^3)$ and choose $C\ge1$ so that
\[
C^{-1}\|v\|_S\le\|v\|_{\mathscr H}\le C\|v\|_S.
\]
First take a finite sum of smooth separated functions
$u=\sum_{\ell=1}^N\prod_{j=1}^k u_{\ell,j}(x_j,v_j)$.
Expanding its squared tensor norm and applying Fubini gives
\[
\|u\|_{S^{\otimes k}}^2
=\sum_{\substack{a_1,\ldots,a_k\in I_2}}
  \|\partial_1^{a_1}\cdots\partial_k^{a_k}u\|_{L^2}^2,
\qquad
\partial_j^{(\alpha,\beta)}=\partial_{x_j}^{\alpha}
                            \partial_{v_j}^{\beta}.
\]
For comparison in coordinate $j$, expand the remaining factors in an orthonormal basis of their current tensor space, writing $u=\sum_n u_n\otimes e_n$. Parseval and the one-particle bound give
\[
C^{-2}\sum_n\|u_n\|_S^2
\le\sum_n\|u_n\|_{\mathscr H}^2
\le C^2\sum_n\|u_n\|_S^2.
\]
Replacing all $k$ norms gives factors $C^{\pm2k}$ for squared norms, proving \cref{eq:H2-Hcal2-equivalence}. Restriction to globally constrained tensors and symmetrization preserve the comparison. Completion extends it and the derivative identity to $\Hcal_k$, with limiting derivatives understood distributionally.
\end{proof}

To sum the energy inequalities over the hierarchy, we must also control growth across particle levels. Geometric activity weights do this while keeping the ratio of adjacent weights compatible with the coupling from level $k+1$ to level $k$.
\begin{definition}\label{def: activity weighted norms}
For $z>0$ and $H=(h_k)_{k\ge1}$ with $h_k\in\Hcal_k$, set
\begin{equation}\label{eq:Ez-def}
\Ecal_z(H)^2:=\sum_{k\ge1}z^{-2k}\|h_k\|_{\Hcal_k}^2,
\qquad \Xbb_z:=\{H:\Ecal_z(H)<\infty\},
\end{equation}
\begin{equation}\label{eq:Dz-def}
\Dcal_z(H)^2:=\sum_{k\ge1}z^{-2k}\|h_k\|_{\Dcal_k}^2,
\qquad \widetilde\Xbb_z:=\{H:\Dcal_z(H)<\infty\}.
\end{equation}
\end{definition}
The absence of a factorial in \cref{eq:Ez-def,eq:Dz-def} keeps adjacent weights compatible with the collision coupling. Factorials in symmetrization and generating functions retain their combinatorial meaning.

\subsection{The energy solution class}\label{subsec:solution-class}
We define energy solutions by smooth approximation of each forced level equation: both the solution and its forcing converge in the norms needed to pass the energy inequality to the limit. This is analogous to Friedrichs' strong-solution construction \cite[Section~4]{Friedrichs1958} and the kinetic graph-space approximation in \cite[Section~6]{AAMN2024}, with the energy and dissipation spaces chosen here for the Boltzmann hierarchy.
\begin{definition}\label{def:Xz-solution}
An energy solution in $\Xbb_z$ on $[0,T]$, or $\Xbb_z$-solution, is a sequence $H=(h_k)$ satisfying the centered hierarchy \cref{eqn: reduced f hat hierachy} distributionally and
\[
H\in L^\infty(0,T;\Xbb_z)\cap L^2(0,T;\widetilde\Xbb_z),\qquad
h_k\in C([0,T];\Hcal_k),\qquad H(0)\in\Xbb_z.
\]
At each level $k$, write $F_k:=\Gamma^{k+1}h_{k+1}\in L^2(0,T;\Dcal_k^*)$. There are curves $u_{k,n}\in C^1([0,T];\Hcal_k)$ taking values in the symmetric, globally constrained smooth core of $A^k$, and approximating forcings $F_{k,n}\in L^2(0,T;\Dcal_k^*)$, such that
\begin{equation}\label{eqn: Xz solution}
\begin{split}
\partial_tu_{k,n}-A^ku_{k,n}&=F_{k,n},\\
u_{k,n}&\longrightarrow h_k\quad\hbox{in }C([0,T];\Hcal_k)\cap L^2(0,T;\Dcal_k),\\
F_{k,n}&\longrightarrow F_k=\Gamma^{k+1}h_{k+1}\quad\hbox{in }L^2(0,T;\Dcal_k^*).
\end{split}
\end{equation}
We take the smooth core to consist of finite sums of symmetrized tensors of smooth periodic functions rapidly decreasing in velocity, with $1-\Pi_g$ applied in every slot. The approximations are made separately at each level; $F_{k,n}$ need not equal $\Gamma^{k+1}u_{k+1,n}$. On $[0,\infty)$ the requirements hold on every finite interval.
\end{definition}

The approximation property transfers the energy inequality to solutions and their differences. General existence remains an assumption; \Cref{prop:mixture-existence} supplies statistical-mixture examples.

We can now verify the link between the density constraints and the spaces just defined. In the centered inversion, macroscopic admissibility makes terms with and without a given particle cancel, removing its global invariant modes.

\begin{lemma}\label{lem: marginal of relative cumulants}
    Let $F(t)$ be macroscopically admissible for almost every $t$. Define $g_k(t)$ by \cref{eqn: inversion formula} and $\widehat f_k(t)=\mu_k^{-1/2}g_k(t)$, and suppose $\widehat f_k(t)\in L^2((\Tbb^3\times\Rbb^3)^k)$ for each $k$ at these times. Then
    \begin{equation}
        \Pi_g^{(j)}\widehat f_k(t)=0,\qquad k\ge1,\quad 1\le j\le k,
        \quad\text{for almost every }t.
    \end{equation}
\end{lemma}
\begin{proof}[Proof of \Cref{lem: marginal of relative cumulants}]
Fix a time $t$ satisfying the hypotheses, a level $k\ge1$, and $j\in[k]$. Let $\phi$ be any element of $\operatorname{span}\{1,v^1,v^2,v^3,|v|^2\}$, and set
\[
c_\phi:=\int_{\Rbb^3}\phi(v)\mu(v)\,dv
=\int_{\Tbb^3\times\Rbb^3}\phi(v)\mu(v)\,dx\,dv.
\]
The equality uses the normalized spatial measure. The subset form of \cref{eqn: inversion formula} is
\[
g_k(t,Z_{[k]})
=\sum_{S\subset[k]}(-1)^{k-|S|}
f_{|S|}(t,Z_S)\prod_{i\in[k]\setminus S}\mu(v_i).
\]
Pair each subset $R\subset[k]\setminus\{j\}$ with $R\cup\{j\}$. For $S=R$, integration against $\phi(v_j)\,dx_j\,dv_j$ produces
\[
(-1)^{k-|R|}c_\phi\,f_{|R|}(t,Z_R)
\prod_{i\in[k]\setminus(R\cup\{j\})}\mu(v_i).
\]
For $S=R\cup\{j\}$, macroscopic admissibility gives
\begin{equation}\label{eq:Jphi-admissibility invoked}
\int_{\Tbb^3\times\Rbb^3}
\phi(v_j)f_{|R|+1}(t,Z_R,z_j)\,dx_j\,dv_j
=c_\phi f_{|R|}(t,Z_R),
\end{equation}
including $R=\varnothing$ by the convention $f_0=1$. The resulting term has the opposite sign $(-1)^{k-|R|-1}$. Thus each pair cancels and
\[
\int_{\Tbb^3\times\Rbb^3}\phi(v_j)g_k(t,X_k,V_k)\,dx_j\,dv_j=0.
\]
Since $\widehat f_k=\mu_k^{-1/2}g_k$, it follows that
\[
\int_{\Tbb^3\times\Rbb^3}
\phi(v_j)\mu^{1/2}(v_j)\widehat f_k(t,X_k,V_k)\,dx_j\,dv_j=0.
\]
Each $\varphi_i/\mu^{1/2}$ is a collision invariant, including $(|v|^2-3)/\sqrt6$. Hence every coefficient in \cref{eqn: hierarchy null projection j} vanishes, proving $\Pi_g^{(j)}\widehat f_k=0$.
\end{proof}

\subsection{Contraction, decay, and statistical solutions}\label{subsec:main-statements}
Fix the adapted norm above with the coefficients supplied by \Cref{prop: Lambda K L hypothesis}, and let $\lambda,c_0>0$ be its coercivity and dissipation-comparison constants in \cref{eqn: A microscopic coersive}. Let $C_\Gamma$ be the pair-collision constant of \Cref{prop: Gamma energy pairing}; the hierarchy bound in \Cref{prop: Gamma energy pairing k} has coefficient $\sqrt{2}C_\Gamma$. The proofs of these operator estimates appear in \Cref{sec:operator-estimates}; all three constants are independent of the particle number.

For small activity, dissipation absorbs the collision coupling, giving contraction in the fixed energy space.

\begin{theorem}\label{thm: main thm}
Under \Cref{A: B regularity}, set
\begin{equation}\label{eq:z-small-in-main-thm}
z_\ast:=\frac{\lambda}{\sqrt{2}C_\Gamma},\qquad 0<z<z_\ast,\qquad
\eta_z:=c_0(\lambda-\sqrt{2}C_\Gamma z)>0.
\end{equation}
Let $H(0)\in\Xbb_z$. There is at most one energy solution ($\Xbb_z$-solution in the sense of \Cref{def:Xz-solution}) of the centered hierarchy \cref{eqn: reduced f hat hierachy} with these initial data. Every such solution on $[0,T]$ satisfies
\begin{equation}\label{eq:exp-decay-main}
\Ecal_z(H(t))\le e^{-\eta_z t}\Ecal_z(H(0)),\qquad 0\le t\le T.
\end{equation}
\end{theorem}
The estimate controls the same $\Xbb_z$ norm throughout any interval of existence, with no additional smallness assumption on $\Ecal_z(H(0))$. If a global macroscopically admissible density solution exists and its rescaled centered coefficients form an $\Xbb_z$-solution, it is unique in this class and converges to the factorized Maxwellian as quantified below.

The fixed-radius estimate controls the whole hierarchy at one rate. Retaining the factor $k$ in the tensor dissipation gives the stronger conclusion below: the activity radius shrinks, and the same initial energy controls decay proportional to each level.

\begin{theorem}\label{thm:theorem-2}
Under the hypotheses of \Cref{thm: main thm}, put $r(t)=z e^{-\eta_z t}$. Then
\begin{equation}\label{eq:moving-radius}
\Ecal_{r(t)}(H(t))\le\Ecal_z(H(0)),\qquad 0\le t\le T.
\end{equation}
Consequently, for every $k\ge1$,
\begin{equation}\label{eq:thm2-main}
\|h_k(t)\|_{\Hcal_k}\le z^k e^{-k\eta_z t}\Ecal_z(H(0)).
\end{equation}
Both theorems apply to differences of solutions. In particular,
\begin{equation}\label{eq:difference-contraction}
\Ecal_{r(t)}(H^{(1)}(t)-H^{(2)}(t))
\le\Ecal_z(H^{(1)}(0)-H^{(2)}(0)).
\end{equation}
\end{theorem}

Statistical mixtures provide correlated examples of the energy class. We form tensor powers of each one-particle solution and average over a probability law. To justify the required smooth approximation, we use the globally constrained space $\mathcal V$ with norm
\[
\|u\|_{\mathcal V}^2:=\sum_{|\alpha|+|\beta|\le3}
\|\langle v\rangle\partial_x^\alpha\partial_v^\beta u\|_2^2,
\qquad \Pi_gu=0.
\]
\begin{proposition}\label{prop:mixture-existence}\label{rem:factorized-existence}
Let $(\Omega,\pi)$ be a probability space. Suppose $f_\omega=\mu+\mu^{1/2}\widehat f_\omega\ge0$ are global one-particle Boltzmann solutions with the mass, momentum, and energy of $\mu$, and that, for each $T<\infty$, the map
\[
\omega\longmapsto \widehat f_\omega\in C([0,T];\mathcal V)\cap C^1([0,T];\Hcal)
\]
is strongly measurable. If $\|\widehat f_\omega(0)\|_{\Hcal}\le\rho<z<z_\ast$ almost surely, then
\begin{equation}\label{eq:mixture-hierarchy}
f_k(t)=\int_\Omega f_\omega(t)^{\otimes k}\,d\pi(\omega),
\qquad h_k(t)=\int_\Omega \widehat f_\omega(t)^{\otimes k}\,d\pi(\omega)
\end{equation}
define a global hierarchy in $\Pbb_{macro}$ whose sequence $H=(h_k)$ is an $\Xbb_z$-solution. In particular,
\begin{equation}\label{eq:mixture-activity-bounds}
\begin{split}
\sup_{t\ge0}\Ecal_z(H(t))^2&\le\frac{\rho^2}{z^2-\rho^2},\\
\int_0^\infty\Dcal_z(H(t))^2\,dt
&\le\frac{\rho^2}{2(\lambda-\sqrt{2}C_\Gamma \rho)z^2(1-\rho^2/z^2)^2}.
\end{split}
\end{equation}
Dirac laws recover factorized solutions, and the law may also be continuous. Existence and the stated regularity of the one-particle solutions are inputs.
\end{proposition}
The proof is given at the end of \Cref{sec: Energy Inequality}, after the operator and energy estimates on which it relies.

The correlations are already visible in a two-point initial mixture. Take a nonzero real $a\in C^\infty(\Tbb^3)$ with $\int a\,dx=0$, and set $\widehat f_0=\varepsilon a\mu^{1/2}$. For sufficiently small $\varepsilon>0$, the densities $f_\pm(0)=\mu(1\pm\varepsilon a)$ are positive and have the conserved moments of $\mu$. Give their product hierarchies equal weight initially. Then
\[
f_1(0)=\mu,\qquad
f_2(0,z_1,z_2)=\mu(v_1)\mu(v_2)
\bigl[1+\varepsilon^2a(x_1)a(x_2)\bigr]\ne\mu^{\otimes2}.
\]
Moreover, $h_k(0)=0$ for odd $k$ and $h_k(0)=\widehat f_0^{\otimes k}$ for even $k$. Thus, with $r=\|\widehat f_0\|_{\Hcal}/z<1$,
\[
\Ecal_z(H(0))^2=\sum_{m\ge1}r^{4m}=\frac{r^4}{1-r^4}.
\]
The initial hierarchy prefactor is of order $\varepsilon^2$, although each branch perturbation is of order $\varepsilon$. The hard-sphere construction below supplies the two global flows. These cancellations concern the initial data; the nonlinear flows need not remain opposite perturbations.

To construct such global branches for hard spheres, we use Jang--Jiang \cite[Theorem~2.1 and (2.13)]{JangJiang2009}. Set
\[
\|u\|_{\mathcal W}^{2}:=
\sum_{|\alpha|\le9}\|\partial_x^\alpha u\|_2^2+
\sum_{|\alpha|+|\beta|\le8}
\|\langle v\rangle^2\partial_x^\alpha\partial_v^\beta u\|_2^2.
\]
Fixing their scaling parameter at $1/4$ and their collision kernel at $B/4$, with $\widehat f=f_{1/4}/4$ in their perturbation notation, the cited theorem gives global solutions of our unscaled one-particle equation for nonnegative data $f_0=\mu+\mu^{1/2}\widehat f_0$ with $\Pi_g\widehat f_0=0$ and sufficiently small $\|\widehat f_0\|_{\mathcal W}$.

By \Cref{lem:one-particle-regularity}, the uniform $\mathcal W$ bound gives $\widehat f,R(\widehat f)\in L^\infty(0,T;\mathcal V)$, where $R(u):=Au+\Gamma[u\otimes u]$. The distributional equation and initial trace give $\widehat f\in W^{1,\infty}(0,T;\mathcal V)$, hence continuity in $\mathcal V$. The local Lipschitz bound makes $R(\widehat f)$ continuous in $H^2$, hence $\widehat f\in C^1([0,T];\Hcal)$ by conservation of global moments. Finite mixtures, including the two-point example, satisfy \Cref{prop:mixture-existence} when each initial branch is small enough that $\|\widehat f_0\|_{\Hcal}<z$.

\section{Operator estimates}\label{sec:operator-estimates}
To close the energy estimate, the dissipation from $A^k$ and the coupling through $\Gamma^{k+1}$ must be measured in matching norms. We first tensorize one-particle coercivity, then bound the coupling in the resulting dual dissipation space with a constant independent of the level.

\subsection{One-particle coercivity and tensorization}\label{subsec:coercivity}
The starting point is the hypocoercive estimate of Mouhot--Neumann in the adapted norm fixed above. On the complement of the global invariant modes, the estimate supplies both dissipation and a comparison with the energy.

\begin{proposition}\label{prop: Lambda K L hypothesis}
Under \Cref{A: B regularity}, the coefficients in \cref{eqn: Hcal 2} can be fixed so that there are constants $\lambda,c_0>0$ with
\begin{equation}\label{eqn: A microscopic coersive}
\langle Au,u\rangle_{\Hcal}\le-\lambda\|u\|_{\Dcal}^2,
\qquad \|u\|_{\Dcal}^2\ge c_0\|u\|_{\Hcal}^2,
\end{equation}
on a smooth rapidly decreasing core with $\Pi_gu=0$, and on its graph closure whenever the pairings are defined.
\end{proposition}
\begin{proof}
The cutoff hard-potential verification in \cite[Section~5.4]{MN06} establishes hypotheses H1$'$--H2$'$--H3 for $0\le\gamma\le1$, nonzero nonnegative $C^1$ angular kernels, and weight comparable to $\nu$. Theorem~3.1, using its quadratic forms (3.5), (3.8) and the lower-order induction, gives precisely a norm of the form \cref{eqn: Hcal 2} and its coercivity on $\ker\Pi_g$. The second estimate follows from $\nu\ge\nu_0$ and the one-particle upper norm comparison $\|u\|_{\mathscr H}\le C_{\rm eq}\|u\|_{H^2}$, with $c_0=\nu_0/C_{\rm eq}^2$. Thus the adapted construction supplies coercivity, while norm comparison supplies the lower bound for dissipation.
\end{proof}

Applying this estimate in each particle coordinate preserves the one-particle constants. Summing over the coordinates supplies the factor $k$ used in the levelwise decay argument.

\begin{proposition}\label{prop: A gap estimate}
For $u\in\Hcal_k$ in the smooth core,
\begin{equation}\label{eq:Aj_gap}
\langle A_ju,u\rangle_{\Hcal^{\otimes k}}
\le-\lambda D_{k,j}(u)^2.
\end{equation}
Consequently,
\begin{equation}\label{eq:tensor-coercivity}
\langle A^ku,u\rangle_{\Hcal_k}\le-\lambda\|u\|_{\Dcal_k}^2,
\qquad \|u\|_{\Dcal_k}^2\ge c_0 k\|u\|_{\Hcal_k}^2.
\end{equation}
The constants are the one-particle constants of \cref{eqn: A microscopic coersive}.
\end{proposition}
\begin{proof}
Fix $j$ and expand the spectator coordinates in an orthonormal basis of $\Hcal^{\otimes(k-1)}$. Every coefficient $u_n$ lies in the one-particle space $\Hcal$, so Parseval and \Cref{prop: Lambda K L hypothesis} give
\[
\langle A_ju,u\rangle=\sum_n\langle Au_n,u_n\rangle_{\Hcal}
\le-\lambda\sum_n\|u_n\|_{\Dcal}^2=-\lambda D_{k,j}(u)^2.
\]
The same expansion gives $D_{k,j}(u)^2\ge c_0\|u\|_{\Hcal_k}^2$; summing over $j$ proves both claims. Since the basis expansion leaves the active coordinate untouched, each coefficient retains its global projection constraint. Approximation extends the estimates where needed.
\end{proof}

\subsection{A collision estimate for correlated pairs}
The coupling requires a bound on arbitrary correlated pair functions. We separate velocity integration from the spatial diagonal restriction, beginning with a weighted velocity estimate that also handles derivatives of the Maxwellian: for a velocity multiindex $\ell$ with $|\ell|\le3$, set $\psi_\ell:=\partial_v^\ell\mu^{1/2}$ and define
\[
\Gamma_{\rm loc}^{(\ell)}G(v):=\int_{\Rbb^3\times\Sbb^2}B(v-w,\sigma)\psi_\ell(w)
[G(v',w')-G(v,w)]\,dw\,d\sigma.
\]
Here $(v',w')$ is the post-collisional pair in \cref{eqn: post collision config} for $(v,w)$. Collisions conserve $|v|^2+|w|^2$, so the Maxwellian satisfies
\[
\mu(v)\mu(w)=\mu(v')\mu(w').
\]
Consequently the gain coefficient $\mu(v)^{-1/2}\mu(v')^{1/2}\mu(w')^{1/2}$ in \cref{eqn: Gamma def} equals $\mu(w)^{1/2}$, the same weight as in the loss term. Thus $\Gamma_{\rm loc}=\Gamma_{\rm loc}^{(0)}$, and the family above includes the Gaussian weights arising upon differentiation.

\begin{proposition}\label{prop: Gamma bounded locally}
For every velocity multiindex $\ell\in\Nbb_0^3$ with $|\ell|\le2$, there is $C_\ell<\infty$ such that, for every measurable $G:\Rbb^6\to\Rbb$ with $\int[\nu(v)+\nu(w)]|G(v,w)|^2\,dv\,dw<\infty$,
\begin{equation}\label{eqn: Gamma bounded}
\|\Gamma_{\rm loc}^{(\ell)}G\|_{L^2_v(\nu^{-1})}
\le C_\ell\left(\int_{\Rbb^6}[\nu(v)+\nu(w)]|G(v,w)|^2\,dv\,dw\right)^{1/2}.
\end{equation}
Equivalently, for every $q\in L^2(\Rbb^3,\nu(v)\,dv)$, the ordinary pairing $\int\Gamma_{\rm loc}^{(\ell)}G(v)q(v)\,dv$ is bounded in absolute value by the right-hand side times $\|q\|_{L^2_v(\nu)}$.
\end{proposition}
\begin{proof}
Write the gain and loss pairings separately and apply Cauchy--Schwarz with collision measure $B(v-w,\sigma)|\psi_\ell(w)|\,dv\,dw\,d\sigma$. In either case the factor containing $q(v)$ is bounded by
\[
\int B|\psi_\ell(w)||q(v)|^2\,dv\,dw\,d\sigma
\le C_\ell\int\nu(v)|q(v)|^2\,dv,
\]
because a polynomial times a Gaussian has finite moments and $\nu(v)\asymp\langle v\rangle^\gamma$. The loss factor containing $G$ is bounded by the squared weighted norm of $G$ on the right of \cref{eqn: Gamma bounded}, since $\psi_\ell$ is bounded and
$|v-w|^\gamma\le C[\nu(v)+\nu(w)]$.
For the gain factor, the pre/post-collisional change of variables preserves the measure $B\,dv\,dw\,d\sigma$; discard the bounded transformed factor $|\psi_\ell(w)|$, integrate the angle, and use the same inequality in the post-collisional variables. This proves the pairing bound, hence \cref{eqn: Gamma bounded} by weighted $L^2$ duality. No factorization or symmetry of $G$ was used.
\end{proof}

The velocity bound leaves the two spatial variables independent. To complete the collision operator, we must restrict them to the diagonal; the following trace estimate shares the required spatial regularity between the two slots.

\begin{lemma}\label{lem: split trace}
Let $U$ be a Hilbert space. If $a,b,r\ge0$, $r\le a,b$, and $a+b-r>3/2$, then
\begin{equation}\label{eqn:trace-split-L2}
\|G(x,x)\|_{H^r_x(U)}\le C_{a,b,r}\|G\|_{H^a_xH^b_y(U)}.
\end{equation}
\end{lemma}
\begin{proof}
For Fourier coefficients, the diagonal coefficient at $m$ is $\sum_n\widehat G(n,m-n)$. Weighted Cauchy--Schwarz and Plancherel reduce the claim to
\[
\sup_{m\in\Zbb^3}\langle m\rangle^{2r}
\sum_n\langle n\rangle^{-2a}\langle m-n\rangle^{-2b}<\infty.
\]
For $|m|\ge2$, split into $|n|\le|m|/2$, $|m-n|\le|m|/2$, and the remaining region. The first two sums are bounded using $b\ge r$ and $a\ge r$, respectively, together with the strict inequality $a+b-r>3/2$; if an exponent equals $3/2$, its logarithmic lattice sum is absorbed by that strict inequality. In the remaining region, dyadic annuli give a bound $C\langle m\rangle^{3-2a-2b+2r}$. The finitely many small $m$ are controlled by $a+b>3/2$. The same Cauchy--Schwarz argument works in $U$.
\end{proof}

The velocity and trace estimates give a direct bound in a weighted
second-order Sobolev norm. For a pair function $G(x,v,y,w)$, define
\[
\|G\|_{2,2,\nu}^2
:=
\sum_{\substack{|\alpha_1|+|\beta_1|\le2\\
                |\alpha_2|+|\beta_2|\le2}}
\int_{(\Tbb^3\times\Rbb^3)^2}
[\nu(v)+\nu(w)]
|\partial_x^{\alpha_1}\partial_v^{\beta_1}
  \partial_y^{\alpha_2}\partial_w^{\beta_2}G|^2
\,dx\,dv\,dy\,dw.
\]
Thus each particle supplies two derivatives in its own phase-space
variables, and either velocity may carry the collision-frequency
weight. No product structure or symmetry is required of $G$.
\begin{proposition}\label{prop: Gamma energy pairing}
For smooth pair functions $G$ with rapid velocity decay,
\[
\|\Gamma G\|_{H^2(\nu^{-1})}
\le C\|G\|_{2,2,\nu}.
\]
Here $H^2(\nu^{-1})$ is the weighted Sobolev norm in \cref{eqn: Hs Lambda 1} with $\nu$ replaced by $\nu^{-1}$. The operator extends continuously to pair functions with finite
$\|\cdot\|_{2,2,\nu}$ norm. In the adapted energy inner product there
is a constant $C_\Gamma$ such that, for $u\in\mathscr D$,
\begin{equation}\label{eqn: Gamma energy inequality}
|\langle\Gamma G,u\rangle_{\mathscr H}|
\le C_\Gamma
\left(
\|G\|_{\mathscr D\otimes\mathscr H}^2
+\|G\|_{\mathscr H\otimes\mathscr D}^2
\right)^{1/2}
\|u\|_{\mathscr D}.
\end{equation}
Here the pairing is the continuous extension of
\cref{eq:derivative-form}.
\end{proposition}
\begin{proof}
We first separate differentiation in velocity from restriction to
equal positions. Simultaneous translation of both colliding
velocities preserves their difference and commutes with the
post-collisional map. In $\Gamma_{\rm loc}G(v+a)$, substitute
$w\mapsto w+a$ before differentiating in $a$. With
$\psi_\ell=\partial_v^\ell\mu^{1/2}$ and the weighted operators
$\Gamma_{\rm loc}^{(\ell)}$, this gives
\[
\partial_v^\beta\Gamma_{\rm loc}G
=
\sum_{\ell+\beta_1+\beta_2=\beta}
\frac{\beta!}{\ell!\,\beta_1!\,\beta_2!}
\Gamma_{\rm loc}^{(\ell)}
   (\partial_v^{\beta_1}\partial_w^{\beta_2}G).
\]
Only the Gaussian and the input function are differentiated;
the collision kernel is unchanged by the simultaneous translation.
The same identity holds with $x,y$ left as independent parameters.

Fix a term in this sum and let
\[
s_1=2-|\beta_1|,\qquad s_2=2-|\beta_2|,
\qquad r=2-|\beta|.
\]
For $|\beta|\le2$, all three indices are nonnegative. Moreover,
$r\le s_1,s_2$ and
$s_1+s_2-r=2+|\ell|>3/2$.
These are precisely the hypotheses of the spatial trace estimate.
To check the input regularity explicitly, set
\[
K(x,y,v):=
\Gamma_{\rm loc}^{(\ell)}
   (\partial_v^{\beta_1}\partial_w^{\beta_2}G)(x,y,v).
\]
The weighted velocity estimate commutes with spatial derivatives,
so \Cref{prop: Gamma bounded locally} gives
\[
\|K\|_{H_x^{s_1}H_y^{s_2}(L^2_v(\nu^{-1}))}
\le C_\ell
\|\partial_v^{\beta_1}\partial_w^{\beta_2}G
 \|_{H_x^{s_1}H_y^{s_2}(L^2_{v,w}(\nu(v)+\nu(w)))}
\le C_\ell\|G\|_{2,2,\nu}.
\]
The last inequality uses
$s_1+|\beta_1|=s_2+|\beta_2|=2$; hence every input derivative
in this norm is one of the derivatives defining
$\|G\|_{2,2,\nu}$.
Apply \Cref{lem: split trace} with
$U=L^2_v(\nu^{-1})$ to restrict $K$ to $y=x$.
Summing the finitely many terms in the velocity differentiation
identity yields
\[
\|\partial_v^\beta\Gamma G
  \|_{H_x^{2-|\beta|}(L^2_v(\nu^{-1}))}
\le C\|G\|_{2,2,\nu},
\qquad |\beta|\le2.
\]
Summing the squared bounds over $\beta$ proves the
$H^2(\nu^{-1})$ estimate.

The finite derivative form of the adapted inner product now gives,
by weighted Cauchy--Schwarz,
\[
|\langle\Gamma G,u\rangle_{\mathscr H}|
\le
\sum_{a,b\in I_2}|m_{ab}|
\|\partial^a\Gamma G\|_{L^2(\nu^{-1})}
\|\partial^b u\|_{L^2(\nu)}
\le C\|G\|_{2,2,\nu}\|u\|_{\mathscr D}.
\]
Fubini identifies the squared input norm as
\[
\|G\|_{2,2,\nu}^2
=\|G\|_{\mathscr D\otimes H^2}^2
 +\|G\|_{H^2\otimes\mathscr D}^2.
\]
Replacing the unweighted $H^2$ norm in each term by the equivalent
adapted norm proves \cref{eqn: Gamma energy inequality}.
The constant $C_\Gamma$ includes only these fixed one-particle
comparison constants and the two-particle estimates above.
Finally, velocity cutoff, mollification, and finite separated
approximations give density of smooth pair functions in the weighted
mixed Sobolev norm. The estimates therefore define both the
distributional collision output and its adapted energy pairing by
continuous extension.
\end{proof}

The same velocity and trace arguments supply the stronger one-particle regularity used in the statistical examples.
\begin{lemma}\label{lem:one-particle-regularity}
Set $R(u):=Au+\Gamma[u\otimes u]$ and
$\|u\|_{Y_4}^2:=\sum_{|\alpha|+|\beta|\le4}\|\langle v\rangle^2\partial_x^\alpha\partial_v^\beta u\|_2^2$.
For $u,w$ with $\Pi_gu=\Pi_gw=0$, whenever the right-hand sides are finite,
\[
\begin{aligned}
\|R(u)\|_{\mathcal V}&\lesssim\|u\|_{Y_4}+\|u\|_{Y_4}^2,\\
\|R(u)-R(w)\|_{H^2}&\lesssim(1+\|u\|_{\mathcal V}+\|w\|_{\mathcal V})\|u-w\|_{\mathcal V}.
\end{aligned}
\]
\end{lemma}
\begin{proof}
For $|\ell|\le3$ and $s=0,1$, collision-measure Cauchy--Schwarz gives
\[
\|\langle v\rangle^s\Gamma_{\rm loc}^{(\ell)}G\|_2^2
\lesssim\int_{\Rbb^6}
\bigl[\langle v\rangle^{2(s+\gamma)}+\langle w\rangle^{2(s+\gamma)}\bigr]|G(v,w)|^2\,dv\,dw.
\]
Indeed, Gaussian moments bound the first Cauchy--Schwarz factor by $C\langle v\rangle^\gamma$. In the remaining integral, $\langle v\rangle^{2s+\gamma}|v-w|^\gamma\lesssim(1+|v|^2+|w|^2)^{s+\gamma}$; boundedness of $\psi_\ell$ and energy conservation under the pre/post-collisional change of variables give the estimate for gain and loss.
The simultaneous-translation identity above holds through order three, differentiating only the Gaussian and the inputs. For $\ell+\beta_1+\beta_2=\beta$ with $|\beta|\le m$, $m=2,3$, assign spatial orders $a_i=m+1-|\beta_i|$ to the inputs and $r=m-|\beta|$ to the output. Then $r\le a_i$ and $a_1+a_2-r=m+2+|\ell|>3/2$, so \Cref{lem: split trace} applies, using at most $m+1$ derivatives per input. Taking $(m,s)=(3,1)$ bounds the quadratic collision term in $\mathcal V$ by $C\|u\|_{Y_4}^2$; $(m,s)=(2,0)$ gives $\|\Gamma[u\otimes w]\|_{H^2}\lesssim\|u\|_{\mathcal V}\|w\|_{\mathcal V}$, since $\gamma\le1$.
These bounds hold without projection constraints; $Lu=\Gamma[u\otimes\mu^{1/2}]+\Gamma[\mu^{1/2}\otimes u]$ gives the linear bounds. Differentiating $v\cdot\nabla_xu$ produces transport and lower-order commutators, hence $\|v\cdot\nabla_xu\|_{\mathcal V}\lesssim\|u\|_{Y_4}$ and $\|v\cdot\nabla_xu\|_{H^2}\lesssim\|u\|_{\mathcal V}$. Bilinearity proves the difference estimate; collision invariance and periodicity preserve the global constraint.
\end{proof}

\subsection{The uniform hierarchy collision bound}\label{subsec:hierarchy-collision}
We now estimate the full coupling $\Gamma^{k+1}=\sum_{j=1}^k\Gamma_j^{k+1}$. Each summand acts on the pair $(j,k+1)$ and leaves the other coordinates unchanged. Lifting the pair estimate through these coordinates preserves its constant; permutation symmetry controls the sum in the dissipation norms.

\begin{proposition}\label{prop: Gamma energy pairing k}
For symmetric $G\in\Dcal_{k+1}$ and $u\in\Dcal_k$, the pair-collision constant $C_\Gamma$ gives the bound
\begin{equation}\label{eqn: Gamma energy pairing}
|\langle\Gamma^{k+1}G,u\rangle_{\Hcal_k}|
\le \sqrt{2}C_\Gamma \|G\|_{\Dcal_{k+1}}\|u\|_{\Dcal_k}.
\end{equation}
Thus the coupling is a bounded map $\Dcal_{k+1}\to\Dcal_k^*$. Its distributional output has zero global invariant moments in every output slot.
\end{proposition}
\begin{proof}
Expand $G$ and $u$ in the same orthonormal basis $(e_n)$ of $\Hcal^{\otimes(k-1)}$ for the unchanged coordinates: after placing these coordinates first, write $G=\sum_n e_n\otimes G_n$ and $u=\sum_n e_n\otimes u_n$. Apply \Cref{prop: Gamma energy pairing} to $(G_n,u_n)$ in the active pair $(j,k+1)$. Cauchy--Schwarz in $n$ yields
\[
|\langle\Gamma_j^{k+1}G,u\rangle|
\le C_\Gamma\bigl[D_{k+1,j}(G)^2+D_{k+1,k+1}(G)^2\bigr]^{1/2}D_{k,j}(u).
\]
Parseval identifies the sums over the common basis with the tensor norms above, introducing no additional constant. Symmetry makes the $k+1$ quantities $D_{k+1,i}(G)^2$ equal; their sum is $\|G\|_{\Dcal_{k+1}}^2$, so each is $\|G\|_{\Dcal_{k+1}}^2/(k+1)$. The sum below contains $2k$ such contributions:
\begin{equation}\label{eq:collision-symmetry-count}
\sum_{j=1}^k[D_{k+1,j}(G)^2+D_{k+1,k+1}(G)^2]
=\frac{2k}{k+1}\|G\|_{\Dcal_{k+1}}^2.
\end{equation}
Summing the pair bounds and applying Cauchy--Schwarz in $j$, with $\sum_jD_{k,j}(u)^2=\|u\|_{\Dcal_k}^2$, therefore gives
\[
|\langle\Gamma^{k+1}G,u\rangle|
\le C_\Gamma\sqrt{\frac{2k}{k+1}}\,
\|G\|_{\Dcal_{k+1}}\|u\|_{\Dcal_k}
\le\sqrt2 C_\Gamma\|G\|_{\Dcal_{k+1}}\|u\|_{\Dcal_k},
\]
which proves \cref{eqn: Gamma energy pairing}.

For the moment statement, a spectator projection commutes with the active-pair collision operator and annihilates $G$. In the active output slot, set the two positions equal. Symmetry of the pair makes the collision-invariant pairing one half of the integral against the sum of the two invariants, which vanishes by pre/post symmetry and conservation of momentum and energy. After undoing the Maxwellian conjugation, this is precisely $\Pi_g^{(j)}\Gamma_j^{k+1}G=0$ in the distributional moment sense. These statements extend from the smooth core by the collision estimate.
\end{proof}

\section{Energy inequality and theorem proofs}\label{sec: Energy Inequality}
The operator bounds now allow us to compare the loss of energy within each level with the input from the next. After transferring the smooth calculation to the solution class, we sum with geometric weights to prove contraction, then use the level-dependent dissipation to obtain a shrinking radius.

\subsection{Levelwise and weighted energy inequalities}\label{subsec:weighted-energy}
We begin at a fixed level, where the core approximation property permits passage from smooth energy identities to the solution. The same passage allows time weights, which will be needed for the shrinking-radius argument.

\begin{lemma}\label{lem:level-energy}
Every solution in \Cref{def:Xz-solution} satisfies, for $0\le s\le t\le T$,
\begin{equation}\label{eq:level-energy}
\frac12\|h_k(t)\|_{\Hcal_k}^2+\lambda\int_s^t\|h_k(\tau)\|_{\Dcal_k}^2\,d\tau
\le\frac12\|h_k(s)\|_{\Hcal_k}^2
+\int_s^t\langle\Gamma^{k+1}h_{k+1},h_k\rangle\,d\tau.
\end{equation}
The same inequality holds distributionally in time and after multiplication by a nonnegative $C^1$ time weight, with its differentiated-weight term.
\end{lemma}
\begin{proof}
For each approximating curve in \cref{eqn: Xz solution}, differentiate its squared norm, insert $A^ku_{k,n}+F_{k,n}$, and apply \cref{eq:tensor-coercivity}. Integrate in time. Strong convergence in $C\Hcal_k$ and $L^2\Dcal_k$, together with forcing convergence in $L^2\Dcal_k^*$, permits passage to every term. Applying the same calculation with a time weight proves the last assertion.
\end{proof}

Because the collision term couples levels $k$ and $k+1$, summation compares adjacent square-root weights. Geometric weights leave a factor $z$ in this comparison, so small activity lets dissipation absorb the entire coupling.

\begin{proposition}\label{prop: energy inequality}
Let $H$ be an $\Xbb_z$-solution in the sense of \Cref{def:Xz-solution}, and let $0<z<z_\ast=\lambda/(\sqrt{2}C_\Gamma)$. Then, distributionally in time,
\begin{equation}\label{eqn: energy inequality}
\frac12\frac{d}{dt}\Ecal_z(H(t))^2
\le-(\lambda-\sqrt{2}C_\Gamma z)\Dcal_z(H(t))^2.
\end{equation}
More precisely, for every $0\le s\le t\le T$,
\begin{equation}\label{eq:integrated-energy}
\frac12\Ecal_z(H(t))^2+(\lambda-\sqrt{2}C_\Gamma z)\int_s^t\Dcal_z(H(\tau))^2\,d\tau
\le\frac12\Ecal_z(H(s))^2.
\end{equation}
The finite-energy representative at every time is determined by the continuous level functions.
\end{proposition}
\begin{proof}
Write $e_k=\|h_k\|_{\Hcal_k}$ and $d_k=\|h_k\|_{\Dcal_k}$. Multiply \cref{eq:level-energy} by $z^{-2k}$ and sum over $1\le k\le M$. The collision estimate \cref{eqn: Gamma energy pairing} and the identity
\begin{equation}\label{eqn:weight-identity}
z^{-2k}=z\,z^{-(k+1)}z^{-k}
\end{equation}
give the finite-sum bound
\[
\begin{aligned}
\sum_{k=1}^M z^{-2k}
 |\langle\Gamma^{k+1}h_{k+1},h_k\rangle|
&\le\sqrt{2} C_\Gamma z
 \sum_{k=1}^M
 (z^{-(k+1)}d_{k+1})(z^{-k}d_k)\\
&\le\sqrt{2} C_\Gamma z
 \left(\sum_{k=1}^M z^{-2(k+1)}d_{k+1}^2\right)^{1/2}
 \left(\sum_{k=1}^M z^{-2k}d_k^2\right)^{1/2}\\
&\le\sqrt{2} C_\Gamma z\,\Dcal_z(H)^2.
\end{aligned}
\]
The bound is independent of $M$ and belongs to $L^1(0,T)$ by the
solution-class assumption. Consequently the absolute collision
series is integrable, and its partial sums converge in $L^1(0,T)$.
The energy and dissipation sums converge monotonically in the
integrated finite-level inequality. Starting at $s=0$ first gives a
finite energy representative at every $t$; the same argument then
applies to every $s$. This proves \cref{eq:integrated-energy} and its
distributional version.
\end{proof}

\subsection{Fixed-radius decay and uniqueness}\label{sec: main proofs}
\begin{proof}[Proof of \Cref{thm: main thm}]
By exact tensorization,
\begin{equation}\label{eq:Poincare-like}
\Dcal_z(H)^2\ge c_0\sum_{k\ge1}kz^{-2k}\|h_k\|_{\Hcal_k}^2
\ge c_0\Ecal_z(H)^2.
\end{equation}
Consequently \Cref{prop: energy inequality} gives
\[
\frac{d}{dt}\Ecal_z(H(t))^2\le-2\eta_z\Ecal_z(H(t))^2.
\]
Its integrated form and Gronwall's inequality imply \cref{eq:exp-decay-main}. The continuous level representatives and the interval inequality extend the estimate from almost every time to every time.

For uniqueness, let $H^{(1)}$ and $H^{(2)}$ have the same initial data. The exact centered hierarchy is linear, and the core-approximation condition is stable under subtraction. Hence $H^{(1)}-H^{(2)}$ is an $\Xbb_z$-solution with zero initial data. The energy estimate forces every component to vanish. The finite reconstruction \cref{eqn: relative cumulant} then gives uniqueness of the associated density hierarchy. For macroscopically admissible densities, \Cref{lem: marginal of relative cumulants} supplies the global projection constraints.
\end{proof}

\subsection{Shrinking activity radius}\label{subsec:moving-radius-proof}
To obtain the energy at radius $z e^{-\eta_z t}$, we approximate the weights $z^{-2k}e^{2k\eta_z t}$ by replacing $k$ in the exponential with $\min\{k,N\}$. For fixed $N$, the resulting sums are controlled by the known fixed-$z$ bounds; monotone convergence then permits $N\to\infty$.
\begin{proof}[Proof of \Cref{thm:theorem-2}]
Write $e_k=\|h_k\|_{\Hcal_k}$ and $d_k=\|h_k\|_{\Dcal_k}$, and set $\eta=\eta_z=c_0(\lambda-\sqrt{2}C_\Gamma z)$. For a positive integer $N$, introduce the time-dependent weights
\[
w_k^{(N)}(t):=z^{-2k}\exp(2\eta\min\{k,N\}t),\qquad
E_N(t)^2:=\sum_{k\ge1}w_k^{(N)}(t)e_k(t)^2.
\]
For fixed $N$ and $T$, these weights are bounded by $e^{2\eta NT}z^{-2k}$, so both their energy and dissipation sums are already controlled by the fixed-radius solution class. Moreover,
\begin{equation}\label{eq:time-weight-ratios}
\sqrt{\frac{w_k^{(N)}}{w_{k+1}^{(N)}}}\le z,\qquad
\frac{d}{dt}w_k^{(N)}=2\eta\min\{k,N\}w_k^{(N)}\le2\eta k w_k^{(N)}.
\end{equation}
Apply the weighted form of \Cref{lem:level-energy}, first for finitely many levels. Then pass to the infinite sum as in the proof of \Cref{prop: energy inequality}. The collision series is dominated in $L^1$ by $\sqrt{2}C_\Gamma z\sum_kw_k^{(N)}d_k^2$; the differentiated-weight series is dominated by $(2\eta/c_0)\sum_kw_k^{(N)}d_k^2$, using \cref{eq:tensor-coercivity}. We obtain
\begin{align*}
\frac12\frac{d}{dt}E_N(t)^2
&\le-(\lambda-\sqrt{2}C_\Gamma z)\sum_kw_k^{(N)}d_k^2
+\eta\sum_k\min\{k,N\}w_k^{(N)}e_k^2\\
&\le-\left(\lambda-\sqrt{2}C_\Gamma z-\frac\eta{c_0}\right)\sum_kw_k^{(N)}d_k^2=0.
\end{align*}
This is an integrated inequality on $[0,T]$, and $E_N(0)=\Ecal_z(H(0))$. Finally, $w_k^{(N)}(t)$ increases to $(z e^{-\eta t})^{-2k}$ as $N\to\infty$. Monotone convergence gives \cref{eq:moving-radius}. Extracting the $k$th energy summand proves \cref{eq:thm2-main}. Thus the weighted inequalities establish membership at the shrinking radius.
\end{proof}

\subsection{Statistical mixtures}
We verify \Cref{prop:mixture-existence} by first approximating each branch and its tensor powers, then controlling the probability average in both the solution and forcing norms.
\begin{proof}[Proof of \Cref{prop:mixture-existence}]
Fix one branch and write $\widehat f=\widehat f_\omega$. To preserve its equation under smooth approximation, we need continuity of $A$ in the forcing norm. The identity
$L\widehat f=\Gamma[\mu^{1/2}\otimes\widehat f+\widehat f\otimes\mu^{1/2}]$ and \Cref{prop: Gamma energy pairing} give $L:\Dcal\to\Dcal^*$ continuously. Since $v\cdot\nabla_x:\mathcal V\to\Hcal$ is bounded, $A:\mathcal V\to\Dcal^*$ is continuous.

Velocity cutoffs and mollification, followed by $1-\Pi_g$, give smooth-core curves $\widehat f^{(n)}\to\widehat f$ in $C\mathcal V$, with $\partial_t\widehat f^{(n)}\to\partial_t\widehat f$ in $C\Hcal$. Their residuals converge to $\Gamma[\widehat f\otimes\widehat f]$ in $L^2\Dcal^*$. Taking tensor powers gives the required approximation at each level: by the product rule, the residual of $(\widehat f^{(n)})^{\otimes k}$ is a sum with one factor $\partial_t\widehat f^{(n)}-A\widehat f^{(n)}$ and $k-1$ factors $\widehat f^{(n)}$.

To sum these paths over levels and average over branches, we need bounds uniform in $\omega$. Coercivity and the pair estimate give, with $a=\|\widehat f\|_{\Hcal}$ and $d=\|\widehat f\|_{\Dcal}$,
\[
\frac12\frac{d}{dt}a^2\le-(\lambda-\sqrt{2}C_\Gamma a)d^2,
\qquad a(t)\le\rho,\qquad
\int_0^\infty d(t)^2\,dt\le\frac{\rho^2}{2(\lambda-\sqrt{2}C_\Gamma\rho)}.
\]
The last two conclusions follow by continuity from $a(0)\le\rho<z_\ast$. Exact tensorization gives
\[
\|\widehat f^{\otimes k}\|_{\Hcal_k}=a^k,\qquad
\|\widehat f^{\otimes k}\|_{\Dcal_k}^2=ka^{2k-2}d^2.
\]
Jensen's inequality and the geometric series now give \cref{eq:mixture-activity-bounds}.

For each fixed $k,T$, these bounds and \Cref{prop: Gamma energy pairing k} make the path and forcing jointly integrable:
\[
\bigl(\widehat f_\omega^{\otimes k},\Gamma^{k+1}\widehat f_\omega^{\otimes(k+1)}\bigr)
\quad\hbox{is Bochner integrable in}\quad
\bigl(C\Hcal_k\cap L^2\Dcal_k\bigr)\times L^2\Dcal_k^*,
\]
where the time interval is $[0,T]$. Approximate this integral in the product-space norm by finite convex combinations of its values, then approximate each selected tensor path by the smooth-core curves above. A diagonal choice proves \cref{eqn: Xz solution} for the averaged coefficient $h_k$.

Bochner--Fubini and continuity of the collision map justify averaging the distributional equations. Positivity, consistency, and macroscopic admissibility follow from the one-particle identities and common conserved moments. The finite centered inversion gives the second formula in \cref{eq:mixture-hierarchy}.
\end{proof}

\section{Consequences for densities and observables}\label{sec:consequences}
The centered estimates describe decay separately at each order. Reconstruction combines these orders within a density, while observable generating functions collect them across the whole hierarchy; we now quantify both consequences.

\subsection{Reconstruction of the densities}\label{subsec:density-reconstruction}
Although $h_k$ decays at rate $k\eta_z$, the reconstruction of $f_k-\mu_k$ contains every centered order $1\le p\le k$. The density bound therefore has the slowest rate $\eta_z$, with constants accounting for the Maxwellian factors and the reversal of the rescaling.

\begin{corollary}\label{cor:fk-to-muk-kdep}
Let $F(t)$ be a macroscopically admissible density-hierarchy solution whose rescaled centered coefficients satisfy the hypotheses of \Cref{thm: main thm}, and recall $\eta_z=c_0(\lambda-\sqrt{2}C_\Gamma z)>0$. Write $\mathcal S_k=(H^2(\Tbb^3\times\Rbb^3))^{\otimes k}$ for the full mixed Sobolev tensor space. Define the one-particle constants
\[
A_\mu:=\|\mu\|_{H^2},\qquad
C_\mu:=\sup_{0\ne h\in\Hcal}\frac{\|\mu^{1/2}h\|_{H^2}}{\|h\|_{\Hcal}},
\]
where $\Hcal$ is the one-particle space in \cref{eq:one-particle-spaces}. Thus $\|\mu^{1/2}h\|_{H^2}\le C_\mu\|h\|_{\Hcal}$; both constants are finite. For every $k\ge1$,
\begin{align}
\|f_k(t)-\mu_k\|_{\mathcal S_k}
&\le \sum_{p=1}^k\binom{k}{p}A_\mu^{k-p}C_\mu^p z^p e^{-p\eta_z t}\Ecal_z(H(0))\label{eq:cor-kwise-decay-hat}\\
&=\bigl[(A_\mu+C_\mu z e^{-\eta_z t})^k-A_\mu^k\bigr]\Ecal_z(H(0))\nonumber\\
&\le e^{-\eta_z t}\bigl[(A_\mu+C_\mu z)^k-A_\mu^k\bigr]\Ecal_z(H(0)).\label{eq:cor-kwise-single-rate}
\end{align}
\end{corollary}
\begin{proof}
Use the exact identity
\[
f_k-\mu_k=\sum_{p=1}^k\binom{k}{p}\Sym(g_p\otimes\mu_{k-p}),\qquad g_p=\mu_p^{1/2}h_p.
\]
Permutations are isometries on $\mathcal S_k$, so $\Sym$ is non-expansive: its operator norm is at most one. The Hilbert tensor norm gives
\[
\|\Sym(g_p\otimes\mu_{k-p})\|_{\mathcal S_k}
\le A_\mu^{k-p}\|g_p\|_{\mathcal S_p}
\le A_\mu^{k-p}C_\mu^p\|h_p\|_{\Hcal_p}.
\]
Apply \Cref{thm:theorem-2}, then sum the binomial expansion. Since $p\ge1$, replacing $e^{-p\eta_z t}$ by $e^{-\eta_z t}$ proves the last inequality.
\end{proof}

\subsection{Dynamical partition function}\label{sec:partition-function}
To collect observable moments across particle levels, we use the exponential correlation-series convention of the Bogoliubov generating functional; see \cite[Section~3, Proposition~9 and Remark~10]{KKO2006}. We call the resulting time-dependent series the dynamical partition function. Here the coefficients are the hierarchy densities and $\beta$ is a source parameter. The factorial normalization is separate from the geometric weights controlling the hierarchy energy.

\begin{definition}\label{def:partition-function}
For $\beta\in\Rbb$ and a measurable observable $\phi:\Tbb^3\times\Rbb^3\to\Rbb$ with
$C_\phi:=\|\phi\mu^{1/2}\|_{L^2}<\infty$, define, whenever the series is absolutely convergent,
\begin{equation}\label{eq:Z-def}
\mathcal Z(\beta,\mu,\phi;F)
:=\sum_{k=0}^\infty\frac{\beta^k}{k!}\langle\phi^{\otimes k},f_k\rangle,
\qquad f_0=1.
\end{equation}
\end{definition}

Inserting the centered decomposition into the defining series for $\mathcal Z$ separates the Maxwellian contribution from the centered moments. Absolute convergence, proved below, justifies the rearrangement and reduces observable convergence to estimating the remaining factor.

\begin{proposition}\label{prop:Z-cumulant-factorization}
Suppose $F$ is reconstructed by \cref{eqn: relative cumulant} from $H\in\Xbb_z$, with $g_k=\mu_k^{1/2}h_k$ and $g_0=1$. Let $\phi:\Tbb^3\times\Rbb^3\to\Rbb$ be measurable with $C_\phi:=\|\phi\mu^{1/2}\|_{L^2}<\infty$. Then the series below converge absolutely for every $\beta\in\Rbb$, and
\begin{equation}\label{eq:Z-factorization}
\mathcal Z(\beta,\mu,\phi;F)
=e^{\beta\langle\phi,\mu\rangle}\mathcal C(\beta,\phi;G),\qquad
\mathcal C(\beta,\phi;G):=\sum_{p=0}^\infty\frac{\beta^p}{p!}\langle\phi^{\otimes p},g_p\rangle.
\end{equation}
\end{proposition}
\begin{proof}
Our one-particle normalization gives $\|h_p\|_{L^2}\le\|h_p\|_{\Hcal_p}$, hence
\[
|\langle\phi^{\otimes p},g_p\rangle|
\le C_\phi^p\|h_p\|_{\Hcal_p}
\le (C_\phi z)^p\Ecal_z(H),\qquad p\ge1.
\]
The same bound holds with $|g_p|$ and $|\phi|$ in the integrand. Since $\langle|\phi|,\mu\rangle\le C_\phi$, the absolute double series obtained from the reconstruction is bounded by
$e^{|\beta|\langle|\phi|,\mu\rangle}[1+\Ecal_z(H)(e^{|\beta|C_\phi z}-1)]$.
Moreover, symmetry and Fubini give
\begin{equation}\label{eq:phi-sym-identity}
\langle\phi^{\otimes k},\Sym(g_p\otimes\mu_{k-p})\rangle
=\langle\phi^{\otimes p},g_p\rangle\langle\phi,\mu\rangle^{k-p}.
\end{equation}
Insert this into \cref{eq:Z-def}, use $\binom{k}{p}/k!=1/(p!(k-p)!)$, and rearrange the absolutely convergent double series.
\end{proof}

The factorization reduces convergence of $\mathcal Z$ to decay of the centered factor toward one. Summing the levelwise bounds gives a quantitative estimate for every observable in the stated weighted space.

\begin{theorem}\label{thm:Z-convergence}
Let $H$ satisfy the hypotheses of \Cref{thm: main thm}, and let $F(t)$ be reconstructed from $H(t)$ by \cref{eqn: relative cumulant}. For every measurable $\phi:\Tbb^3\times\Rbb^3\to\Rbb$ with $C_\phi:=\|\phi\mu^{1/2}\|_{L^2}<\infty$ and every $\beta\in\Rbb$, put $q:=|\beta|C_\phi z$. Then
\begin{equation}\label{eq:Z-limit}
|\mathcal Z(\beta,\mu,\phi;F(t))-e^{\beta\langle\phi,\mu\rangle}|
\le e^{|\beta|\langle|\phi|,\mu\rangle}(e^q-1)e^{-\eta_z t}\Ecal_z(H(0)).
\end{equation}
\end{theorem}
\begin{proof}
Cauchy--Schwarz and \cref{eq:thm2-main} bound each centered coefficient by
\[
|\langle\phi^{\otimes p},g_p(t)\rangle|
\le(C_\phi z)^p e^{-p\eta_z t}\Ecal_z(H(0)),\qquad p\ge1.
\]
Consequently,
\begin{align}
|\mathcal C(\beta,\phi;G(t))-1|
&\le\sum_{p\ge1}\frac{q^p}{p!}e^{-p\eta_z t}\Ecal_z(H(0))
=\bigl(e^{q e^{-\eta_z t}}-1\bigr)\Ecal_z(H(0))\label{eq:C-series-bound}\\
&\le(e^q-1)e^{-\eta_z t}\Ecal_z(H(0)).\label{eq:C-coarse-bound}
\end{align}
For the last step, set $s=e^{-\eta_z t}\in[0,1]$. The exponential series has nonnegative coefficients, and $s^p\le s$ for $p\ge1$, so
\[
e^{qs}-1=\sum_{p\ge1}\frac{q^p s^p}{p!}
\le s\sum_{p\ge1}\frac{q^p}{p!}=s(e^q-1).
\]
Multiplying \cref{eq:C-coarse-bound} by the Maxwellian factor in \cref{eq:Z-factorization} proves \cref{eq:Z-limit}.
\end{proof}

The same estimate controls the logarithmic generating function. For real $\beta,\phi$, once the bound on $|\mathcal C-1|$ is below $1/2$, one has $\mathcal C>0$ and
\[
|\log\mathcal Z-\beta\langle\phi,\mu\rangle|
=|\log\mathcal C|\le2|\mathcal C-1|.
\]
Thus the corresponding logarithmic generating function also approaches its equilibrium value at rate $\eta_z$ whenever the solution exists globally.

\subsection{Connected observable correlations}\label{subsec:stat-mech-interp}\label{rem:gluing}
The logarithm of the generating function separates connected correlations between distinct particles from products of lower-order moments. Under the hypotheses of \Cref{thm:Z-convergence}, the factorial coefficient bounds in its proof give locally uniform convergence of the series and its derivatives in $\beta$. Since $\mathcal Z(0,\mu,\phi;F(t))=1$, its logarithm is real analytic near zero. Define
\[
\kappa_1(t):=\left.\partial_\beta\log\mathcal Z(\beta,\mu,\phi;F(t))\right|_{\beta=0},\qquad
\kappa_2(t):=\left.\partial_\beta^2\log\mathcal Z(\beta,\mu,\phi;F(t))\right|_{\beta=0}.
\]
Termwise differentiation identifies the first two classical joint cumulants in the partition-lattice convention of Speed \cite{Speed1983}: the mean and the connected correlation between two distinct particles.
\begin{equation}\label{eq:connected-observable-identities}
\kappa_1(t)=\langle\phi,f_1(t)\rangle,\qquad
\kappa_2(t)=\langle\phi^{\otimes2},f_2(t)\rangle-\langle\phi,f_1(t)\rangle^2.
\end{equation}
Writing $E=\Ecal_z(H(0))$ and $a_\phi(t)=C_\phi z e^{-\eta_z t}$, we obtain
\begin{equation}\label{eq:connected-observable-decay}
|\kappa_1(t)-\langle\phi,\mu\rangle|\le a_\phi(t)E,
\qquad |\kappa_2(t)|\le a_\phi(t)^2(E+E^2).
\end{equation}
Indeed, if $b_j=\langle\phi^{\otimes j},g_j(t)\rangle$, the identity
$\log\mathcal Z=\beta\langle\phi,\mu\rangle+\log\mathcal C$
gives $\kappa_1-\langle\phi,\mu\rangle=b_1$ and $\kappa_2=b_2-b_1^2$; the bound $|b_j|\le a_\phi^j E$ proves both estimates.

For an admissible density hierarchy, take $\phi=\mathbf1_A$ for a measurable region $A\subset\Tbb^3\times\Rbb^3$. Then $\kappa_2(t)$ is the excess probability that two distinct particles both lie in $A$ over the product of their one-particle probabilities. Its decay at rate $2\eta_z$ quantifies loss of this correlation. For the mixtures of \Cref{prop:mixture-existence}, absolute convergence gives $\mathcal Z=\mathbb E_\pi\exp(\beta\langle\phi,f_\omega(t)\rangle)$, hence $\kappa_2=\operatorname{Var}_\pi\langle\phi,f_\omega(t)\rangle$: the variance across conditional one-particle means. In the same representation, $h_k=\mathbb E_\pi(\widehat f_\omega^{\otimes k})$ are tensor moments centered at the fixed Maxwellian.

\appendix
\section{Elementary combinatorial identities}\label{sec: Combinatorial Lemmas}
The centered identities in \Cref{subsec:centered-decomposition} follow from two elementary facts: outer symmetrization absorbs an inner average, and an alternating binomial sum selects a single coefficient.

\subsection{Symmetrization and binomial inversion}
When we substitute one centered expansion into another, the nested averages simplify because composing a fixed permutation with every permutation leaves the full average unchanged.

\begin{lemma}\label{lem: sym associative}
    Let $a,b,c \geq 0$ be integers, and let $h_a$ be a symmetric function of $a$ particle coordinates. Then
    \begin{equation*}
        \Sym\Bigl( \Sym\bigl( h_a \otimes \mu_{b} \bigr) \otimes \mu_{c}\Bigr) = \Sym\bigl( h_a \otimes \mu_{b + c}\bigr).
    \end{equation*}
\end{lemma}
\begin{proof}[Proof of \Cref{lem: sym associative}]
Set $k=a+b+c$ and $F=h_a\otimes\mu_b\otimes\mu_c$. Extend each $\tau\in S_{a+b}$ to $S_k$ by leaving the last $c$ labels fixed. For fixed $\tau$, the map $\sigma\mapsto\sigma\circ(\tau\oplus\mathrm{Id}_c)$ is a bijection of $S_k$. Hence the nested average equals
\[
\frac{1}{(a+b)!k!}\sum_{\tau\in S_{a+b}}\sum_{\sigma\in S_k}
\bigl(\sigma\circ(\tau\oplus\mathrm{Id}_c)\bigr)F
=\frac1{k!}\sum_{\rho\in S_k}\rho F=\Sym F.
\]
Combining the Maxwellian factors gives $\Sym F=\Sym(h_a\otimes\mu_{b+c})$.
\end{proof}
After the nested symmetrizations are removed, the inversion reduces to an alternating binomial sum. Its cancellation isolates the coefficient at the desired level.

\begin{lemma}\label{lem: binomial inversion}
    Let $k\geq 1$ be a fixed integer, and let $a_0, \dots, a_k \in \Rbb$. Then
    \begin{equation*}
        \sum_{0\leq \ell \leq p \leq k} (-1)^{k-p} \binom{k}{p} \binom{p}{\ell}a_\ell = a_k.
    \end{equation*}
\end{lemma}
\begin{proof}[Proof of \Cref{lem: binomial inversion}]
Reverse the order of summation and use
$\binom{k}{p}\binom{p}{\ell}=\binom{k}{\ell}\binom{k-\ell}{p-\ell}$.
With $r=p-\ell$, the sum becomes
\[
\sum_{\ell=0}^k a_\ell\binom{k}{\ell}
\sum_{r=0}^{k-\ell}(-1)^{k-\ell-r}\binom{k-\ell}{r}.
\]
The binomial theorem makes the inner sum zero when $\ell<k$ and one when $\ell=k$, leaving $a_k$.
\end{proof}

\subsection{Proofs of the centered identities}
\begin{proof}[Proof of \Cref{prop: Mobius inversion}]
We verify that substituting the forward expansion into the proposed inverse recovers $g_k$. Write
\[
\tilde g_k := \sum_{p=0}^k (-1)^{k-p}\binom{k}{p}\Sym\bigl(f_p\otimes\mu_{k-p}\bigr).
\]
Substitute the forward map \cref{eqn: relative cumulant}, $f_p = \sum_{q=0}^p\binom{p}{q}\Sym(g_q\otimes\mu_{p-q})$:
\[
\tilde g_k
= \sum_{p=0}^k (-1)^{k-p}\binom{k}{p}
  \Sym\!\Bigl(\sum_{q=0}^p\binom{p}{q}\Sym\bigl(g_q\otimes\mu_{p-q}\bigr)\otimes\mu_{k-p}\Bigr).
\]
By linearity of $\Sym$ and \Cref{lem: sym associative},
$\Sym\bigl(\Sym(g_q\otimes\mu_{p-q})\otimes\mu_{k-p}\bigr) = \Sym(g_q\otimes\mu_{k-q})$.
Rearranging the double sum gives
\[
\tilde g_k
= \sum_{q=0}^k \Sym\bigl(g_q\otimes\mu_{k-q}\bigr)
  \underbrace{\sum_{p=q}^k (-1)^{k-p}\binom{k}{p}\binom{p}{q}}_{=\,\delta_{q,k}},
\]
where the inner sum equals $\delta_{q,k}$ by \Cref{lem: binomial inversion}, applied pointwise
(with $a_\ell = \Sym(g_\ell\otimes\mu_{k-\ell})$).
Hence $\tilde g_k = \Sym(g_k\otimes\mu_0) = g_k$,
where $\mu_0 := 1$ (the zero-particle Maxwellian is the constant function~$1$).
\end{proof}

\begin{proof}[Proof of \Cref{lem:Sym-subset-representation}]
Group the permutation average by the subset $S$ occupying the first $p$ slots. Since $h_p$ and the Maxwellian product are symmetric in their own slots, each subset contributes the same function for its $p!(k-p)!$ orderings. Its coefficient in the average is therefore
\[
\frac{p!(k-p)!}{k!}=\binom{k}{p}^{-1}.
\]
Summing over all $p$-element subsets gives \cref{eq:Sym-subset}.
\end{proof}

\section*{Acknowledgments}
I thank Cl{\'e}ment Mouhot and Philip~J. Morrison for discussions at SLMath, Tai-Ping Liu for clarifying the near-vacuum setting of my earlier joint work \cite{AMPT2025}, Jonathan Luk for early discussions, and Orin Yakir for pointing out \cite{Speed1983,Speed1986}. I am grateful to Stanford University for supporting my Szeg\H{o} Assistant Professorship (2024--2026), and to SLMath for hosting my postdoctoral visit during the program \emph{Kinetic Theory: Novel Statistical, Stochastic and Analytical Methods}.

I used AI tools developed by Anthropic and OpenAI for mathematical discussions and assistance with editing this manuscript. I take full responsibility for the mathematical arguments, references, and final text.

\section*{Statements and Declarations}
\noindent\textbf{Funding.} This material is based upon work supported by the National Science Foundation under Grant No.~DMS-2424139, while the author was in residence at the Simons Laufer Mathematical Sciences Institute in Berkeley, California, during the Fall 2025 semester.

\noindent\textbf{Competing interests.} The author has no relevant financial or non-financial interests to disclose.

\noindent\textbf{Data availability.} No datasets were generated or analysed in this theoretical study.


\end{document}